\documentclass{daj}

\usepackage{geometry} 
\usepackage{fullpage} 
\usepackage{setspace}
\usepackage{savesym}
\usepackage{lscape, threeparttable}
\usepackage{afterpage,epsfig,pgfplotstable}
\usepackage{lipsum}

	\usetikzlibrary{automata,topaths}
	\usetikzlibrary{plotmarks}
	\usetikzlibrary{patterns}
	\usetikzlibrary{svg.path}

\newenvironment{changemargin}[3]{%
\begin{list}{}{%
\setlength{\topsep}{#1}%
\setlength{\leftmargin}{#2}%
\setlength{\rightmargin}{#3}%
\setlength{\listparindent}{\parindent}%
\setlength{\itemindent}{\parindent}%
\setlength{\parsep}{\parskip}%
}%
\item[]}{\end{list}}

\usepackage{mathtools}
\mathtoolsset{showonlyrefs}

\newcommand {\mathsym}[1]{{}}
\newcommand {\unicode}[1]{{}}

\usepackage{caption}
\usepackage{array}
\usepackage{diagbox}
\usepackage{graphicx}
\usepackage{tikz}
\usepackage{amssymb}
\usepackage{pdfsync}
\usepackage{hyperref}
\usepackage{verbatim} 
\usepackage{epstopdf}
\usepackage{esint} 
\usepackage{float}
\usepackage{multirow}
\usepackage[figuresright]{rotating}

\usepackage{amssymb}
\usepackage{latexsym}
\usepackage{amsmath}
\usepackage{bbm}
\usepackage{amsthm}

\renewcommand{\d}{\,{\rm d}} 
\newcommand{\sph}[1]{\mathbb{S}^{#1}}

\def\R{\mathbb{R}}
\def\T{\mathbb{T}}
\def\N{\mathbbm{N}}

\def\Z{\mathbb{Z}}
\def\Co{\mathbb{C}}
\def\A{\mathbb{A}}

\def\a{\mathcal{A}}

\def\H{\mathcal{H}}
\def\F{\mathcal{F}}

\newcommand{\la}{\lambda}

\newcommand{\ft}{\widehat}

\newcommand{\al}{\alpha}
\newcommand{\be}{\beta}
\renewcommand{\b}{\mathcal{B}}
\newcommand{\B}{\mathbb{B}}
\newcommand{\om}{\omega}

\renewcommand{\leq}{\leqslant}
\renewcommand{\geq}{\geqslant}

\renewcommand{\frak}{\mathfrak}

\theoremstyle{plain}
\newtheorem{theorem}{Theorem}[section]
\newtheorem{corollary}[theorem]{Corollary}
\newtheorem{definition}[theorem]{Definition}
\newtheorem{proposition}[theorem]{Proposition}
\newtheorem{lemma}[theorem]{Lemma}

\newtheorem{example}[theorem]{Example}
\newtheorem{conjecture}[theorem]{Conjecture}
\newtheorem{prob}{Problem}[section]

\theoremstyle{definition}

\numberwithin{equation}{section}

\dajAUTHORdetails{%
  title = {New Sign Uncertainty Principles}, 
  author = {Felipe Gon\c{c}alves, Diogo  Oliveira e Silva, and Jo\~ao P. G. Ramos},
  plaintextauthor = {Felipe Goncalves, Diogo  Oliveira e Silva, Joao P. G. Ramos},
  keywords = {Dini series, 
Fourier series, 
Fourier transform,
Gegenbauer polynomials, 
Hamming cube, 
Hankel transform,
Hilbert transform, 
Jacobi polynomials, 
linear programming, 
sphere packing, 
spherical design,
spherical harmonics,
uncertainty principle.},
}   

\dajEDITORdetails{%
   year={2023},
   number={9},
   received={7 September 2021},   
   published={21 July 2023},  
   doi={10.19086/da.84266},       
}   

\begin{document}

\begin{frontmatter}[classification=text]

\title{New Sign Uncertainty Principles} 

\author[fg]{Felipe Gon\c{c}alves\thanks{Supported by the Deutsche Forschungsgemeinschaft through the Collaborative Research Center 1060.}}
\author[dos]{Diogo Oliveira e Silva\thanks{Supported by the EPSRC New Investigator Award ``Sharp Fourier Restriction Theory'', grant no.\@ EP/T001364/1.}}
\author[jpgr]{Jo\~ao P. G. Ramos\thanks{Supported by the Deutscher Akademischer Austauschdienst.}}

\begin{abstract}
We prove new sign uncertainty principles which vastly generalize the recent developments of Bourgain, Clozel \& Kahane and Cohn \& Gon\c{c}alves, and apply our results to a variety of spaces and operators. In particular, we establish new sign uncertainty principles for 
Fourier and Dini series, 
the Hilbert transform, 
the discrete Fourier and Hankel transforms, 
spherical harmonics, 
and Jacobi polynomials, 
among others. 
We present numerical evidence highlighting the relationship between the discrete and continuous sign uncertainty principles for the Fourier and Hankel transforms, which in turn are connected with the sphere packing problem via linear programming. 
Finally, we explore some connections between the sign uncertainty principle on the sphere and spherical designs.\end{abstract}
\end{frontmatter}

\section{Introduction}

The uncertainty principle, discovered by W.\@ Heisenberg in 1927, is one of the cornerstones of quantum mechanics.
It can be expressed via Heisenberg's inequality:
\[\inf_{a,b\in\mathbb{R}} \int_{-\infty}^\infty (x-a)^2 |f(x)|^2 \textup d x \int_{-\infty}^\infty (\xi-b)^2 |\widehat{f}(\xi)|^2\textup d\xi \geq\frac{\|f\|_{L^2(\mathbb R)}^4}{16\pi^2},\]
where $\widehat f$ denotes the Fourier transform of $f$. 
 This estimate reflects the fact that the Fourier transform of a highly localized function must necessarily be widely dispersed in frequency space. 
 Six years later, G.\@ H.\@ Hardy developed a more refined theory in this respect, and in particular established the following result:
If there exist $a,b>0$, such that the estimates $f(x)=O(e^{-a\pi x^2})$, $\widehat{f}(\xi)=O(e^{-b\pi \xi^2})$ hold, then $f\equiv 0$ whenever $ab>1$, and $f$ must coincide with a polynomial multiple of the Gaussian function $e^{-a\pi x^2}$ if $ab=1$.
 Thus the uncertainty inequalities of Heisenberg and Hardy respectively explore, in a quantitative way, the notions of {\it concentration} around the origin and {\it decay} at infinity; see \cite{FS97} for further details. 
 
In 2010, motivated by applications to number theory,  Bourgain, Clozel \& Kahane \cite{BCK10} investigated an analogue of the uncertainty principle, where the notions of concentration and decay are replaced by that of {\it nonnegativity}. 
To describe it precisely, consider the following setting.
Given $d\geq 1$, a function $f:\mathbb R^d\to\mathbb R$ is said to be {\it eventually nonnegative} if $f(x)\geq 0$ for all sufficiently large $|x|$.
In this case, consider the  quantity
\[r(f):=\inf\{r> 0: f(x)\geq 0 \text{ if } |x|\geq  r\},\] 
which corresponds to the radius of the last sign change of $f$.
Normalize the  Fourier transform,
\begin{equation} \label{def:fouriertrans}
\widehat{f}(\xi)=\int_{\R^d} f(x)e^{-2\pi i \langle x,\xi\rangle} \d x, 
\end{equation}
where $\langle\cdot,\cdot\rangle$ represents the usual inner product in $\mathbb R^d$.
Let $\mathcal A_+(d)$ denote the set of functions $f:\R^d\to\R$ which are not identically zero and satisfy the following conditions: 
\begin{itemize}
\item $f\in L^1(\R^d)$, $\widehat{f}\in L^1(\R^d)$, and $\widehat{f}$ is real-valued (i.e.\@ $f$ is even);
\item $f$ is eventually nonnegative while $\widehat{f}(0)\leq 0$;
\item $ \widehat f$ is eventually nonnegative while ${f}(0)\leq 0$.
\end{itemize}
The product $r(f) r(\widehat{f})$ is invariant under rescaling, and becomes a natural quantity to consider. 
In this setting, the authors of \cite{BCK10} estimated the quantity 
\begin{equation}\label{defAd+}
{\mathbb A}_+(d):=\inf_{f\in\mathcal A_+(d)\setminus\{\bf 0\}} \sqrt{r(f) r(\widehat{f})}.
\end{equation}
In particular, it is shown in \cite[Th\'eor\`eme 3.1]{BCK10} that ${\mathbb A}_+(d)$ is bounded from below, and that in fact it grows linearly with the square root of the dimension. 

Very recently, Cohn \& Gon\c{c}alves \cite{CG19} discovered a complementary uncertainty principle which is connected with the linear programming bounds of Cohn \& Elkies \cite{CE03} for the sphere packing problem. 
To describe it precisely, let $\mathcal A_-(d)$ denote the set of functions $f:\mathbb R^d\to\mathbb R$ which satisfy the following conditions: 
\begin{itemize}
\item $f\in L^1(\R^d)$, $\widehat{f}\in L^1(\R^d)$, and $\widehat{f}$ is real-valued (i.e.\@ $f$ is even);
\item $f$ is eventually nonnegative while $\widehat{f}(0)\leq 0$;
\item $-\widehat f$ is eventually nonnegative while ${f}(0)\geq 0$.
\end{itemize}
In a similar spirit to \cite{BCK10}, the authors of \cite{CG19} showed that the quantity
\begin{equation}\label{defAd-}
\mathbb A_-(d):=\inf_{f\in\mathcal A_-(d)\setminus\{\bf 0\}} \sqrt{r(f) r(-\widehat{f})}
\end{equation}
is bounded from below, and that in fact it grows linearly with $\sqrt{d}$. 
We shall refer to the boundedness of the quantities defined in \eqref{defAd+}, \eqref{defAd-} as the {\it $\pm 1$ uncertainty principles}; see \S\ref{moreback} below (in particular, the statement of Theorem \ref{thm:PreThm}) for further information. 
Our first main result consists in the following generalization of the $\pm 1$ uncertainty principles.

\begin{theorem}[Operator Sign Uncertainty Principle]\label{thm:OSUP}
Let $X,Y$ be two arbitrary measure spaces, equipped with positive measures $\mu,\nu$, respectively. Let $\F\subseteq L^1(X,\mu)\times L^1(Y,\nu)$ be a given family of pairs of functions.
Assume that there exist real numbers $p,q>1$ and $a,b,c>0$, such that, for every $(f,g) \in \F$, 
\begin{itemize}
\item $\|g\|_{L^\infty(Y,\nu)}\leq a\|f\|_{L^1(X,\mu)};$
\smallskip

\item $\|g\|_{L^q(Y,\nu)}\leq b\|f\|_{L^p(X,\mu)};$
\smallskip

\item $\|f\|_{L^p(X,\mu)}\leq c\|g\|_{L^q(Y,\nu)};$
\smallskip

\item $\int_X f\d\mu\leq 0, \, \,\, \int_Y g\d\nu\leq 0$.
\smallskip

\end{itemize}
Then, for every nonzero $(f,g) \in \F$, the following inequality holds:
\begin{equation}\label{eq:OSUP}
\mu(\{x\in X: f(x)<0\})^{\frac{1}{p'}}\nu(\{y\in Y:  g(y)<0\})^{\frac{1}{q}}\geq a^{-1}b^{-\frac{q'}q}(2c)^{-q'},
\end{equation}
where $p'=p/(p-1)$ denotes the exponent conjugate to $p$, and similarly for $q'$.
\end{theorem}

The designation {\it Operator Sign Uncertainty Principle} derives from the fact that the family $\F$ is usually defined in terms of a given invertible operator $T:L^p(X,\mu)\to L^q(Y,\nu)$, 
i.e., it is often the case that $\F=\{(f,T(f)): f\in \mathcal S\}$, for some $\mathcal S\subseteq L^p(X,\mu)$. 
For instance, if for\footnote{Henceforth we shall  use the letter $s$ to denote a sign from $\{+,-\}$ and, by a slight but convenient abuse of notation, we will sometimes identify the signs $\{+,-\}$ with the integers $\{+1,-1\}$.} $s\in \{+,-\}$ we let
$$\F_s=\{(f,s \ft f): f,s\ft f\in L^1(\R^d) \text{ and both  eventually nonnegative}\},$$ 
then the hypotheses of Theorem \ref{thm:OSUP} are satisfied with $p=q=2$ and $a=b=c=1$. Since $f(x),s\ft f(\xi)\geq 0$ for $|x|\geq r(f),|\xi|\geq r(s\ft f)$, respectively, it follows that
\begin{equation}\label{eq:BCKGCasConseq}
\frac{1}{16}\leq
 |\{x\in \R^d: f(x)<0\}||\{\xi\in \R^d: s\ft f(\xi)<0\}| \leq
  |B_1^d|^2 r(f)^dr(s\ft f)^d.
 \end{equation}
Here, $|E|$ represents the Lebesgue measure of a given set $E\subseteq\R^d$, and $B_1^d\subseteq\R^d$ denotes the unit ball centered at the origin.
In turn, estimate \eqref{eq:BCKGCasConseq} immediately implies the aforementioned $\pm1$ uncertainty principles of Bourgain, Clozel \& Kahane and Cohn \& Gon\c{c}alves. 

Theorem \ref{thm:OSUP} opens the door to a variety of novel sign uncertainty principles of interest, as evidenced by the many examples explored in  \S \ref{sec:polyspaces}, \S \ref{sec:discspaces}, \S \ref{sec:ConvolutionOps} below, which we shall introduce as further main results of the present article.  
For instance, in \S \ref{sec:polyspaces} we establish a sign uncertainty principle for Fourier series. In \S \ref{sec:discspaces}, we describe some discrete sign uncertainty principles, which in the limit seem to converge back to the continuous $\pm 1$ uncertainty principles. In \S \ref{sec:ConvolutionOps}, we discuss sign uncertainty principles for certain convolution operators on spaces of bandlimited functions, including the Hilbert transform. 
These connections are entirely new, and can potentially find many applications in several different branches of mathematics.

Motivation for our second main result comes from letting $Y=\N:=\{0,1,2,3,\ldots\}$ in Theorem \ref{thm:OSUP}, and taking $\F$ to be the family of pairs $(f,s\ft f)$, for some chosen sign $s\in\{+,-\}$, where $\ft f:\N\to\R$ is the coefficient sequence obtained by expanding $f$ in some orthonormal basis. We shall derive a result that applies to a wide class of metric measure spaces, which we proceed to describe.  Let $X=(X,{d},\lambda)$ be a metric measure space, with a distance function $d:X\times X\to [0,\infty)$, and a probability measure $\lambda$. Further consider the space $L^2(X,\lambda)$ of square-integrable, {\it real-valued} functions $f:X\to\R$, which we will simply denote by $L^2(X)$ if no confusion arises. Given $x\in X$ and $r>0$, let $B(x,r):=\{y\in X: {d}(x,y)\leq r\}$.

\begin{definition}[Admissible space]\label{def:AdmSp} 
The space $(X,{d},\lambda)$ is admissible if there exists an orthonormal basis
$\{\varphi_n:X\to\R\}_{n\in\N}$ of $L^2(X)$ and a fixed point\footnote{It may be useful to think of $\frak 0$ as the {\it origin of $X$} with respect to the basis $\{\varphi_n\}_{n\in\N}$.} $\frak 0\in X$, such that $\varphi_0\equiv 1$, and,  for every $n\in\N$,
\begin{equation}\label{eq:OriginVsInfty}
\varphi_n(\frak 0):=\lim_{r\to 0^+}\frac1{\lambda(B(\frak 0, r))}\int_{B(\frak 0, r)} \varphi_n\,\d\lambda=\|\varphi_n\|_{L^\infty(X)}<\infty.
\end{equation}
\end{definition}

\begin{definition}[The $\mathcal A_s(X)$-cone]\label{def:AsCone}
Let $s\in\{+,-\}$. 
Let $(X,{d},\lambda)$ be an admissible space, for which $\{\varphi_n\}_{n\in\N}$ is an orthonormal basis of $L^2(X)$ satisfying \eqref{eq:OriginVsInfty} for some $\frak 0 \in X$.
Then $\mathcal A_s(X)$ consists of all square-integrable functions $f:X\to\R$, such that:
\begin{itemize}
\item If $f=\sum_{n=0}^\infty \widehat f(n)\varphi_n$ then 
\begin{equation}\label{eq:L1summable}
\sum_{n=0}^\infty |\widehat f(n)| \|\varphi_n\|_{L^\infty(X)}<\infty;
\end{equation}
\item $\widehat f(0)\leq 0;$
\item $\{s\widehat f(n)\}_{n\in\N}$ is eventually nonnegative while $sf(\frak 0)\leq 0$.
\end{itemize}
\end{definition}

Here $\widehat f(n)=\langle f,\varphi_n\rangle_{L^2(X)}=\int_X f\varphi_n\d\lambda$. 
Note that $\mathcal A_s(X)\subseteq L^1(X)$ since $L^2(X)\subseteq L^1(X)$. 
From \eqref{eq:L1summable}, it also follows that $\widehat f\in \ell^1(\N)$ if $f\in\mathcal A_s(X)$, simply because $\|\varphi_n\|_{L^\infty(X)} \geq \|\varphi_n\|_{L^2(X)}=1$.
Since the series $\sum_{n=0}^\infty \widehat f(n)\varphi_n$ converges absolutely and uniformly, the function $f$ would coincide $\lambda$-almost everywhere with a continuous function {\it if} each $\varphi_n$ were continuous. While this is the case for most of our applications, the latter continuity property is not strictly necessary to make sense of the value of a given $f\in\mathcal A_s(X)$ at $\frak 0$.
Indeed, in the current setting, one can easily show that $\frak 0$ is a Lebesgue point of $f$, and invoke \eqref{eq:L1summable} to define $f(\frak 0)$ as follows:
$$
f(\frak 0):=\lim_{r\to 0^+}\frac1{\lambda(B(\frak 0, r))}\int_{B(\frak 0, r)} f\,\d\lambda=\sum_{n=0}^\infty \widehat f(n)\varphi_n(\frak o)
$$

Given $r_1,r_2\in [0,\infty)$, we write $r_1\sim r_2$ if $\la(B(\frak o,r_1))=\la(B(\frak o,r_2))$, or equivalently if $B(\frak o,r_1)=B(\frak o,r_2)$ up to $\la$-null sets.
One easily checks that $\sim$ defines an equivalence relation on $[0,\infty)$, and that each equivalence class is an interval which contains its infimum. 
Let $\mathcal{R}:=\{\inf I: I \in [0,\infty)/ \sim\}$.  
Given $f\in\mathcal A_s(X)$, we  define\footnote{Definition \eqref{eq:defrf} turns out to be more adequate than merely taking the infimum over all $r\geq 0$. Indeed, let $X=\N$, with $d(n,m):=|n-m|$ and counting measure $\la$. Then $\mathcal{R}=\N$, and $r(f;X)$ coincides with the unique integer $m\geq 1$, for which $f(m-1)<0$ but $f(n)\geq 0$ for all $n\geq m$.} the following quantities:
\begin{align}
r(f;X)&:=\inf\{r \in \mathcal{R}: f(x)\geq 0 \text{ for } \la \text{-a.e.  } x\in X \text{ such that } d(x,\frak o)\geq r\};\label{eq:defrf}\\
k(s\widehat f)&:=\min\{k\geq 1: s\widehat{f}(n) \geq 0 \text{ if }  n\geq k\}.\label{eq:defkfhat}
\end{align} 
In fact, throughout the paper, given a sequence $\{a_n\}_{n=0}^N\subset \R$ with $N<\infty$ or $N=\infty$, we will more generally write 
$$
k_a = k(a) = \min\{k\geq 0: a_n\geq 0 \text{ if } n\geq k\}.
$$ 
Note that $r(f;X)$ can be $+\infty$, or equal to the smallest $r_0>0$ for which $X\subseteq B(\frak o,r_0)$.
On the other hand, if $f$ is nonzero, then $r(f;X)>0$ as long as $\la(\{\frak o\})=0$, for otherwise $f\geq 0$ ($\la$-a.e.), which contradicts $\ft f(0)\leq 0$. 
Moreover, $s\ft f(n)$ cannot be nonnegative for all $n\geq 0$, for otherwise
$$
0\leq   \sum_{n=0}^\infty s\ft f(n)\varphi_n(\frak o) =sf(\frak o)\leq 0,
$$
and therefore $\ft f(n)=0$, for all $n\geq 0$, which is absurd because $f$ is nonzero. 
We also have that $k(-\ft f)\geq 2$, for otherwise
\begin{align*}
f(x)-\ft f(0) =  \sum_{n=1}^\infty \ft f(n) \varphi_n(x) \geq  \sum_{n=1}^\infty \ft f(n) \varphi_n(\frak o)  = f(\frak o) -\ft f(0),
\end{align*}
whence $f(x) \geq f(\frak o) \geq 0$ for all $x\in X$, which is absurd because $\ft f(0)\leq 0$ and $f$ is nonzero. On the other hand,  it might be the case that $k(\ft f)=1$ (e.g.\@ take $f\equiv-1$); but if $\ft f(0)=0$, then it is easy to see that $k(\ft f)\geq 2$ as well.

We are now ready to state our second main result.
  
\begin{theorem}[Orthonormal Sign Uncertainty Principle]\label{thm:MSUP}
Let $s\in\{+,-\}$.
Let $(X,d,\lambda)$ be an admissible space, for which $\{\varphi_n\}_{n\in\N}$ is an orthonormal basis of $L^2(X)$ satisfying \eqref{eq:OriginVsInfty} for some $\frak 0\in X$. 
Then, for every nonzero $f\in\mathcal A_s(X)$, the following inequality holds:
\begin{equation}\label{eq:MSUP1}
\lambda(\{x\in X: f(x)<0\}) \sum_{\substack{n\geq 0:\\ s\ft f(n)<0}} \|\varphi_{n}\|_{L^\infty(X)}^2\geq\frac1{16}.
\end{equation}
In particular, it holds that
\begin{equation}\label{eq:MSUP}
\lambda(B(\frak o, r(f;X))) \sum_{n=0}^{ k(s\widehat f)-1} \|\varphi_{n}\|_{L^\infty(X)}^2\geq\frac1{16}.
\end{equation}
\end{theorem}

 Theorems \ref{thm:OSUP} and \ref{thm:MSUP} are not entirely unrelated: for instance, the latter easily follows from the former (with a lower bound which possibly differs from $\frac1{16}$) in the special case when the orthonormal basis satisfies $\sup_{n\in\N}\|\varphi_n\|_{L^\infty(X)}<\infty$. 
 If the space $L^2(X)$ is finite dimensional, then a corresponding version of Theorem \ref{thm:MSUP} holds; we omit the obvious statement, but note that the proof is exactly the same.
Consequences of Theorem \ref{thm:MSUP} to a variety of settings will be explored in \S \ref{sec:polyspaces}.\footnote{For most applications, we will limit ourselves to  the formulation given by \eqref{eq:MSUP}, see Theorems \ref{thm:Sphere}, \ref{thm:Jacobi}, \ref{thm:dini}, \ref{thm:Hilbert}, \ref{thm:Hankel} below, but the reader should bear in mind that the more general formulation given by \eqref{eq:MSUP1} likewise holds in all of these results.} In particular, we establish a sign uncertainty principle for spherical harmonics in \S \ref{sec:Sph}. It turns out that, in the case of the unit sphere $\mathbb S^{d-1}\subseteq\R^d$,  the zero set of a minimizer to the restricted problem on a finite dimensional subspace $V={\rm span }\{\varphi_n\}_{n=0}^N$ exhibits natural geometric structure.
In particular, we shall see how to relate this zero set to the set of cosine distances of certain spherical designs.

\subsection{Further Background} \label{moreback} 
We briefly expand on the history of previous work which inspired the present paper, and its connections to our main results. 
The initial lower and upper bounds for ${\mathbb A}_+(d)$ of Bourgain, Clozel \& Kahane \cite{BCK10} were subsequently sharpened by Gon\c{c}alves, Oliveira e Silva \& Steinerberger \cite{GOSS17}. Cohn \& Gon\c{c}alves \cite{CG19} then discovered that the sign uncertainty principle is connected with the linear programming bounds for the sphere packing problem, and exploited this connection to prove that ${\mathbb A}_+(12)=\sqrt{2}$. Crucially, they realized the applicability of the powerful machinery devised by Viazovska \cite{Vi17} in her solution to the eight-dimensional sphere packing problem to construct eigenfunctions of the Fourier transform via certain 
 Laplace transforms of modular forms. 
To understand this connection in greater depth, we shall briefly discuss the upper bounds on sphere packings via linear programming from the groundbreaking work of Cohn \& Elkies \cite{CE03}. 
Let $\a_{LP}(d)$ denote the set of functions $f:\R^d\to\R$, which satisfy the following conditions:
\begin{itemize}
\item $f\in L^1(\R^d)$, $\widehat{f}\in L^1(\R^d)$, and $\widehat{f}$ is real-valued (i.e.\@ $f$ is even);
\item $-f$ is eventually nonnegative while $\widehat{f}(0)= 1$;
\item $ \widehat f$ is nonnegative and ${f}(0)= 1$.
\end{itemize}
In \cite[Theorem 3.2]{CE03} it is shown that, given any sphere packing $\mathcal P\subseteq \R^d$ of congruent balls, its upper density $\bar{\delta}(\mathcal P)$ satifies
\begin{equation}\label{packingbound}
\bar{\delta}(\mathcal P) \leq r(-f)^d |B_{\frac12}^d|,
\end{equation}
for any $f\in \a_{LP}(d)$. Therefore the quantity 
$$
\mathbb A_{LP}(d):=\inf_{f\in\mathcal \a_{LP}(d)} r(-f)
$$
becomes of interest. High precision numerical data indicated that the upper bound \eqref{packingbound} agrees with the packing density of the hexagonal, $E_8$, and Leech lattices in dimensions $2,8$, and $24$, respectively. In a celebrated breakthrough, Viazovska  \cite{Vi17} found the magical function $f$ realizing equality in \eqref{packingbound} when $d=8$, thereby proving optimality of the $E_8$-lattice packing and showing that $\A_{LP}(8)=\sqrt{2}$. Shortly thereafter, Cohn, Kumar, Miller, Radchenko \& Viazovska \cite{CKMRV17} used similar methods to prove the optimality of the Leech lattice when $d=24$, thereby showing that $\A_{LP}(24)=2$. An elementary geometric argument reveals that the hexagonal packing is optimal if $d=2$ (see e.g.\@ \cite{H00}), but the corresponding magical function is yet to be discovered. Cohn \& Gon\c{c}alves \cite{CG19} later noticed that the $-1$ uncertainty principle described in the previous section underpins the construction in dimensions $d\in\{8,24\}$. The connection is simple to describe: 
If $f\in \a_{LP}(d)$, then $\ft f-f\in \a_-(d)$ and $r(\ft f-f)\leq r(-f)$, and therefore $\A_-(d)\leq \A_{LP}(d)$. In \cite{CG19}, the authors performed extensive numerical calculations, producing compelling evidence towards the following conjecture,\footnote{Conjecture \ref{conj:ALPA-1conj} is equivalent to \cite[Conjecture 7.2]{CE03}; the equivalence was proven in \cite{CG19}.} which if proved would establish a precise mathematical link between the sign uncertainty principle and the sphere packing problem, and clarify the constructions in \cite{CKMRV17,Vi17}.

\begin{conjecture}\label{conj:ALPA-1conj}
$\A_{LP}(d)=\A_-(d)$, for every $d\geq 1$.
\end{conjecture}

Indeed, one can extract the $-1$ eigenfunctions from \cite{CKMRV17, Vi17}, and then use Poisson-type summation formulae for the $E_8$ and Leech lattices (in the same way as the Eisenstein series  $E_6$ was used to prove optimality in \cite{CG19}) in order to conclude that $\A_{LP}(8)=\A_-(8)=\sqrt{2}$ and $\A_{LP}(24)=\A_-(24)=2$. Cohn \& Elkies \cite{CE03} further showed that $\A_{LP}(1)=1$, and that the function $f(x)=(1-|x|)_+$ is optimal; from their proof, one can easily derive that $\A_{-}(1)=1$, and that a corresponding minimizer is given by the function $x\mapsto (\ft f-f)(x)=\frac{\sin^2(\pi x)}{(\pi x)^2}-(1-|x|)_+$. Together with $\A_+(12)=\sqrt{2}$ (recall \cite{CG19}), these constitute a complete list of dimensions $d$ for which $\A_\pm(d), \A_{LP}(d)$ are known. 
 From the possible equality in \eqref{packingbound} for the hexagonal packing when $d=2$, Cohn \& Elkies \cite{CE03} further conjectured 
 that $\A_{LP}(2)=(\frac43)^{\frac14}$. Therefore one should also expect that $\A_-(2)=(\frac43)^{\frac14}$. 
 
 \begin{conjecture}\label{conj:A-1ALP2dconj}
$\A_{LP}(2)=\A_-(2)=(\frac43)^{\frac14}$.
\end{conjecture}
 
As a consequence of our new sign uncertainty principle for the discrete Fourier transform (see \S \ref{sec:DFT}, \S\ref{numerics:DFT} below), we now have numerical evidence pointing towards following conjecture.

\begin{conjecture}\label{conj:A+11dconj}
$\A_+(1)<0.555$. Moreover, any minimizer for $\A_+(1)$ vanishes identically in a sequence of nonempty intervals after the last sign change (see Figure \ref{figure:1}).
\end{conjecture}

To the best of our knowledge, these are the only dimensions for which even a guess of the actual solution exists, all other dimensions remaining for the most part entirely mysterious. We believe that solving Conjectures \ref{conj:A-1ALP2dconj} or \ref{conj:A+11dconj} would require brand new techniques, which could potentially be applied to other dimensions, and open windows of possibilities. Even though the exact answer is not known, or even conjectured, in any other dimension $d\notin\{1,2,8,12,24\}$, it has been established that radial minimizers exist in all dimensions, and that such minimizers must necessarily vanish at infinitely many radii greater than $\A_+(d)$. This was shown in \cite[Theorem 4]{GOSS17} for the $+1$ uncertainty principle, and the technique was later \cite{CG19} adapted to handle the $-1$ uncertainty principle. The following result summarizes the state-of-the-art knowledge of minimizers for the $\pm 1$ uncertainty principles.

\begin{theorem}[\cite{BCK10, CG19, GOSS17}]\label{thm:PreThm}
Let $d\geq 1$. Then the following two-sided inequalities hold:
\begin{align}
& \frac1{\sqrt{2\pi e}}\leq \frac{\mathbb A_+(d)}{\sqrt{d}}\leq \frac{1}{\sqrt{2\pi}}+o_d(1);\label{eq:SUPRd+}\\
& \frac1{\sqrt{2\pi e}}\leq \frac{\mathbb A_-(d)}{\sqrt{d}}\leq 0.3194...+o_d(1).\label{eq:SUPRd-}
\end{align}
Moreover, for each $s\in\{+,-\}$ and $d\geq 1$, there exists a radial function $f\in\mathcal A_s(d)\setminus\{{\bf 0}\}$, such that $\widehat{f}=sf$, $f(0)=0$, $r(f)=\mathbb A_s(d)$. Any such function must vanish at infinitely many radii greater than $\mathbb A_s(d)$.
\end{theorem}

The number $0.3194...$ in \eqref{eq:SUPRd-} is derived from the classical upper bounds of Kabatiansky \& Levenshtein \cite{KL78} for the sphere packing problem.
Indeed, the construction in \cite{CZ14} reveals how the same bound can be obtained via linear programming, whence $\A_{LP}(d)\leq (0.3194...+o_d(1))\sqrt{d}$.
The upper bound in \eqref{eq:SUPRd-} then follows from the aforementioned estimate $\A_-(d)\leq \A_{LP}(d)$. 
In spite of the distinct upper bounds in \eqref{eq:SUPRd+}, \eqref{eq:SUPRd-}, it is conjectured in \cite{CG19} (with strong numerical evidence) that there exists a constant $c>0$, for which  ${\mathbb A_+(d)} \sim {\mathbb A_-(d)}\sim c\sqrt{d}$, as $d\to \infty$. Moreover, there are reasons to believe that $c$ might not be too far from $0.3194$; indeed, recent numerical results in the framework of the modular bootstrap in conformal field theory \cite{ACHLT20} suggest that $c=\tfrac1\pi$. The structural statement in Theorem \ref{thm:PreThm} (concerning the double roots of the minimizers) stem from a seemingly new observation concerning Hermite polynomials, which relates their pointwise values to linear flows on the torus $\T^d$, and extends to other families of orthogonal polynomials; see \cite{GOSS19} for further applications of this idea. The proof of \cite[Theorem 4]{GOSS17} can easily be adapted to show that minimizers for $\A_{LP}(d)$  exist, and must also have infinitely many double roots.
Finally, some equivalent formulations of the $\pm 1$ uncertainty principles, and mass concentration phenomena exhibited by the corresponding minimizing sequences, were the subject of very recent explorations in \cite{GOSR20}. Further related recent results can be found in \cite{CMS19,GVT19}.

\subsection{Outline}
In \S\ref{sec:polyspaces}, we establish sign uncertainty principles for 
spherical harmonics (\S\ref{sec:Sph}),
Jacobi polynomials (\S\ref{sec:Jac}),
Fourier series (\S\ref{sec:Tor}),
and Dini series (\S\ref{sec:diniseries}).
In \S\ref{sec:discspaces}, we establish sign uncertainty principles for 
the discrete Fourier transform (\S\ref{sec:DFT}),
the discrete Hankel transform (\S\ref{sec:DHT}),
and the Hamming cube (\S\ref{sec:Ham}).
In \S\ref{sec:ConvolutionOps}, we establish sign uncertainty principles for 
convolution kernels in bandlimited function spaces (\S\ref{sec:CKBL}),
the Hilbert transform of bandlimited functions (\S\ref{sec:convolution}),
and the Hankel transform (\S\ref{sec:hankeltransform}).
The main results are proved in \S\ref{sec:PfThms}.
Finally, in \S\ref{sec:numerics}, we present our numerical findings related to
the discrete Fourier transform (\S\ref{numerics:DFT}),
and the discrete Hankel transform (\S\ref{numerics:DHT}).

\section{Sign Uncertainty for Classical Orthogonal Systems}
\label{sec:polyspaces}

\subsection{Spherical Harmonics}\label{sec:Sph}
Let $\sph{d-1}=\{\omega\in\R^d: |\omega|=1\}$ denote the unit sphere, equipped with the geodesic distance ${d}_g:\sph{d-1}\times\sph{d-1}\to[0,\pi]$,
${d}_g(\omega,\nu):=\arccos(\langle\omega,\nu\rangle),$
and normalized surface measure $\bar\sigma$, induced from the ambient space $\R^d$ in the natural way and satisfying $\bar\sigma(\sph{d-1})=1$.
The special orthogonal group $\textup{SO}(d)$ consists of all $d\times d$ orthogonal matrices of unit determinant, and acts transitively on the unit sphere $\sph{d-1}$.
The vector space of spherical harmonics on $\sph{d-1}$ of degree $n$, denoted $\mathcal H_n^{d}$,
consists of restrictions to $\sph{d-1}$ of real-valued harmonic polynomials on $\R^d$ which are homogeneous of degree $n$.
The spaces $\mathcal H_n^{d}$ are mutually orthogonal and span $L^2(\sph{d-1})=L^2(\sph{d-1},\bar\sigma)$,
\[L^2(\sph{d-1})=\bigoplus_{n=0}^\infty\mathcal H_n^{d}.\]
Let $h_n:=\text{dim } \mathcal H_n^{d}$, and denote the north pole by $\eta=(0,\ldots,0,1)\in \sph{d-1}$. 

\begin{definition}[Signed basis]\label{def:AdmBasis}
An orthonormal basis $\{Y_{n,j}\in\mathcal H_n^{d}: n\in \N, j= 1,\ldots,h_n\}$ of $L^2(\sph{d-1})$ is signed if:
\begin{itemize} 
\item $Y_{n,j}(\eta)\geq 0$, for every $n\in \N,j=1,2,\ldots,h_n;$
\item $Y_{n,j}(\eta)> 0$, for every $j=1,2,\ldots,h_n$, provided $n$ is sufficiently large. 
\end{itemize}
\end{definition}
 
A signed basis for $L^2(\sph{d-1})$ can be constructed as follows.
Given a continuous function $f:\sph{d-1}\to\R$, let $Z(f):=\{\omega\in\sph{d-1}: f(\omega)=0\}$ denote its zero set.
Start with an arbitrary basis $\mathcal Y=\{Y_{n,j}\in \mathcal H_n^{d}: n\in\N, j=1,2,\ldots,h_n\}$ of $L^2(\sph{d-1})$, and consider the corresponding zero set,
\[\mathcal Z(\mathcal Y):=\bigcup_{n=0}^\infty\bigcup_{j=1}^{h_n} Z(Y_{n,j}).\]
Since $\bar\sigma(\mathcal Z(\mathcal Y))=0$, we can  find a rotation $\rho\in \textup{SO}(d)$ such that  $\rho(\eta)\notin \mathcal Z(\mathcal Y)$.
Therefore there exists a sequence of signs $\{s_{n,j}\}\subseteq \{+,-\}^\N$, for which $\{s_{n,j} Y_{n,j}\circ\rho:n\in\N, j=1,2,\ldots,h_n\}$ is a signed basis for $L^2(\sph{d-1})$.

Henceforth, we fix a signed orthonormal basis $\{Y_{n,j}:n\in\N,j=1,2,\ldots,h_n\}$ of $L^2(\sph{d-1})$.
Any real-valued, square-integrable function $f:\sph{d-1}\to\R$ can be expanded as follows:
\begin{equation}\label{eq:ExpandFSph}
f=\sum_{n=0}^\infty \sum_{j=1}^{h_n} \widehat{f}(n,j) Y_{n,j},
\end{equation}
where $\widehat{f}(n,j)=\int_{\sph{d-1}} f(\omega) Y_{n,j}(\omega)\d\bar\sigma(\omega)$.

\begin{definition}[The $\mathcal{B}_s(\sph{d-1})$-cone]\label{def:AsSph} 
Let $s\in\{+,-\}$.
Then $\mathcal{B}_s(\sph{d-1})$ consists of all continuous functions $f:\sph{d-1}\to\R$, such that:
\begin{itemize}
\item $\widehat{f}(0,1)\leq0;$
\item  $\{s\widehat f(n,j):n\in\N,j=1,2,\ldots,h_n\}$ is eventually nonnegative 
while $sf(\eta)\leq 0$.
\end{itemize}
\end{definition}
 Given $f\in\mathcal{B}_s(\sph{d-1})$, set 
\begin{align*}
\theta(f)&:=\inf\{\theta\in(0,\pi]: f(\omega)\geq 0 \text{ if }d_g(\omega,\eta)\geq  \theta\};\\
 k(s\widehat f)&:=\min\{k\geq 1: s\widehat f(n,j)\geq 0 \text{ if }  n\geq k\},
\end{align*}
and define the quantity
\begin{equation}\label{eq:B_sSph}
{\mathbb B}_s(\sph{d-1}):=\inf_{f\in \mathcal B_s(\sph{d-1})\setminus\{{\bf 0}\}} (1-\cos(\theta(f)))^\frac12 k(s\widehat f),
\end{equation}
which is estimated by our next result.

\begin{theorem}\label{thm:Sphere}
Let $s\in\{+,-\}$ and $d\geq 2$. 
Then the following estimates hold:
\begin{equation} \label{eq:2sidedSph}
{\mathbb B}_s(\sph{d-1}) \geq  {\frac{2\Gamma(\frac{d+1}2)^{\frac2{d-1}}}{(4e^{\frac1{12}})^{\frac{2}{d-1}}(d^2-1)^\frac12}},
\end{equation}
\begin{equation}\label{eq:upperbdsBsph}
{\mathbb B}_+(\sph{d-1})\leq \sqrt{2}, \ \ \text{ and }\ \ \ {\mathbb B}_-(\sph{d-1})\leq 2\sqrt{2}.
\end{equation}
\end{theorem}

\noindent {\it Remark.}
Since $(1-\cos\theta)^\frac12 =\sqrt{2}\sin \frac{\theta}2 \approx \theta$ if $0\leq \theta\leq \pi$, a similar uncertainty principle would be obtained if $(1-\cos(\theta(f)))^\frac12$ were  replaced by $\theta(f)$ in \eqref{eq:B_sSph}. We made this choice with a view towards identity \eqref{eq:spheregegenid} below, which would otherwise be merely a two-sided inequality instead of an equality.
Further note that by Stirling's formula we have
$$
{\frac{2\Gamma(\frac{d+1}2)^{\frac2{d-1}}}{(4e^{\frac1{12}})^{\frac{2}{d-1}}(d^2-1)^\frac12}} = e^{-1}+O(d^{-1}{\log d}),
$$
which is in sharp contrast with the Euclidean (noncompact) case where $\A_s(d) \approx \sqrt{d}$.\\

The proof of Theorem \ref{thm:Sphere} involves Gegenbauer polynomials, which are particular instances of Jacobi polynomials, discussed in \S\ref{sec:Jac} below.
As with most results in this section, Theorem \ref{thm:Sphere} ultimately boils down to a special case of a more general result from \S\ref{sec:Jac}.
More precisely, the proof of the lower bound \eqref{eq:2sidedSph} proceeds in two steps. Firstly, via a zonal symmetrization procedure, we may assume the existence of an eventually nonnegative sequence of coefficients $\{a_n\}_{n\in\N}$, for which
\[f(\omega)=\sum_{n=0}^\infty a_n C^{{d/2-1}}_n(\langle\omega,\eta\rangle ).\]
Here,
$C^{{d/2-1}}_n$ denotes the Gegenbauer polynomial of degree $n$ and order $\frac d2-1$; see \eqref{eq:JacobivsGegen} below. Secondly, the map
$g(x) \mapsto g(\langle\omega,\eta\rangle)$
defines a bijection between the set $\b_s(I;\tfrac{d-3}{2},\frac{d-3}{2})$ 
from Definition \ref{def:AsJac} below and the set of functions in $\b_{s}(\sph{d-1})$ which are invariant under rotations that fix the north pole.  Consequently, the following identity holds:
\begin{equation}\label{eq:spheregegenid}
{\mathbb B}_s(\sph{d-1})^2={\mathbb B}_s\left([-1,1];\tfrac{d-3}{2},\tfrac{d-3}{2}\right),
\end{equation}
where the right-hand side is defined in \eqref{eq:BsJac} below. 
Therefore Theorem \ref{thm:Sphere} will ultimately follow from Theorem \ref{thm:Jacobi};
see \S\ref{sec:PfThmSph} for details.

\begin{definition}[The class $\b^0_s(\sph{d-1})$]\label{def:b0def}
Let $s\in\{+,-\}$ and $d\geq 2$. 
Then $\b^0_s(\sph{d-1})$ consists of all functions $f\in\b_s(\sph{d-1})$ which are invariant under rotations that fix the north pole $\eta$, and satisfy  $f(\eta)=0$. 
\end{definition}

Further define the quantity
\[{\mathbb B}^0_s(\sph{d-1}):=\inf_{f\in \mathcal B^0_s(\sph{d-1})\setminus\{{\bf 0}\}} (1-\cos(\theta(f)))^\frac12 k(s\widehat f).\]
The following result is a direct consequence of \eqref{eq:spheregegenid} and Proposition \ref{prop:WLOG}  below.

\begin{proposition}\label{prop:B0B}
Let $s\in\{+,-\}$ and $d\geq 2$.  Then ${\mathbb B}^0_s(\sph{d-1})={\mathbb B}_s(\sph{d-1})$.
\end{proposition}

For the remainder of this section, we investigate polynomials in $\b_s^0(\sph{d-1})$ which are {\it optimal} in the following sense. 

\begin{definition}[$s$-optimal polynomial in $\b^0_s(\sph{d-1})$]\label{def:localopt}
Let $s\in\{+,-\}$ and $d\geq 2$. 
A polynomial $f \in \b^0_s(\sph{d-1})$ is locally $s$-optimal if there exists $\delta>0$, such that 
\[
(1-\cos(\theta(f)))^\frac12 k(s\widehat f) < (1-\cos(\theta(h)))^\frac12 k(s\widehat h),
\]
for any polynomial $h\in\b_s^0(\sph{d-1})$ satisfying $\deg(h)\leq \deg(f)$ and $0<\inf_{c>0} \|f-ch\|_{L^\infty(\sph{d-1})}<\delta$. 
The polynomial $f$ is said to be globally $s$-optimal if one can take $\delta=+\infty$.
\end{definition}

\subsubsection{\underline{Connections with Spherical Designs}} \label{sec:Designs}
A fundamental tool employed in the solutions of the sphere packing problem in  ~8 and ~24 dimensions \cite{Vi17,CKMRV17} and of the +1-uncertainty principle in ~12 dimensions \cite{CG19} is the Poisson summation formula associated with certain modular forms; recall the discussion in \S \ref{moreback}. Poisson summation is often used to extract sharp lower bounds, and to access information about the root location of the conjectural minimizer. On the sphere $\sph{d-1}$, the role of  Poisson summation seems to be played by {\it spherical designs}; see \cite{BB09} for an excellent introduction to this topic.

Let us introduce some terminology.
A finite subset $\Omega\subseteq \sph{d-1}$ is called a {\it spherical $t$-design} if, 
for every polynomial $f:\sph{d-1}\to\R$ of degree at most $t$,
$$
\int_{\sph{d-1}} f(\omega)\d \bar \sigma(\omega) = \frac{1}{\#\Omega}\sum_{\om \in \Omega} f(\om ).
$$
We say that $\Omega$ has $m$ distances if the {\it set of cosine distances},
\[\al(\Omega):=\{\langle\om,\om'\rangle:\om,\om'\in \Omega,\, \om\neq \om'\},\]
is such that $\#\al(\Omega)=m$; in this case, we write $\al(\Omega)=\{\al_m<\al_{m-1}<\ldots<\al_1\}$. Note that necessarily $t\leq 2m$, for otherwise the nonnegative, nonzero function 
$$
f(\om)=(1-\langle\omega,\omega_1\rangle)\prod_{j=1}^{m} (\langle\omega,\omega_1\rangle -\al_j)^2,\,\,\,(\omega_1\in\Omega)
$$
would have zero average on $\sph{d-1}$. 
Moreover, if $t=2m$, then $\Omega$ cannot contain a pair of antipodal points, for otherwise $\al_m=-1$, and the function
$$
g(\om)=(1-\langle\omega,\omega_1\rangle^2)\prod_{j=1}^{m-1} (\langle\omega,\omega_1\rangle -\al_j)^2
$$
would have zero average on $\sph{d-1}$, which is again impossible. 

Delsarte, Goethals \& Seidel \cite{DGS77} showed that, if $\Omega\subseteq\sph{d-1}$ is a spherical $t$-design, then
\begin{equation}\label{eq:DGSlb}
\#\Omega \geq \binom {d+\lfloor t/2\rfloor-1}{\lfloor t/2\rfloor} + \binom {d+\lceil t/2\rceil-2}{\lceil t/2\rceil-1}.
\end{equation}
A spherical $t$-design $\Omega\subseteq\sph{d-1}$ is said to be {\it tight} if equality holds in \eqref{eq:DGSlb}. 
 It is also shown in \cite{DGS77} that, if $\Omega$ is a spherical $t$-design, then $\Omega$ is tight if and only if $\#\al(\Omega)=\lceil t/2 \rceil$ and $\Omega$ is antipodal if $t$ is odd. 
 
 The regular $(t+1)$-gon is a tight $t$-design on $\sph{1}\subseteq\R^2$, for any $t\geq 1$.
By contrast, tight $t$-designs on $\sph{d-1}$ with $d\geq 3$ are rare. In particular, Bannai \& Damerell \cite{BD79,BD80} established the following: if $d\geq 3$, then tight $t$-spherical designs can only exist when $t \in\{ 1,2,3,4, 5, 7,11\}$. Moreover, modulo isometries: 
if $t=1$, then $\Omega$ consists of a pair of antipodal points; 
if $t=2$, then $\Omega$ is the regular simplex with $d+1$ vertices; 
if $t=3$, then $\Omega=\{\pm e_j\}_{j=1}^d$ is the cross-polytope with $2d$ vertices;
 and if $t=11$, then $d=24$ and $\Omega$ is the set of $196\,560$ minimal vectors of the Leech Lattice. The complete classification of spherical $t$-designs is open for $t\in\{4,5,7\}$, although several examples are known; see \cite[p.~1401]{BB09} and \cite[Table 1]{CK07}. 

\begin{definition}[$s$-optimal spherical design]\label{def:designopt}
Let $s\in\{+,-\}$ and $d\geq 2$. 
Let $\Omega\subseteq \sph{d-1}$ be a tight spherical $t$-design with  
$\al(\Omega)=\{\al_m<\al_{m-1}<\ldots<\al_1\}$, where $m=\lceil t/2 \rceil$. 
For $m\geq 2$, let $a=1$ if $\al_m=-1$, $a=2$ if $\al_m>-1$,  and define the polynomial
\begin{equation}\label{eq:defPolyP}
P(\om):=(x-1)(x -\al_m)^a(x-\al_1)\prod_{j=2}^{m-1} (x -\al_j)^2,\,\text{ where }x=\langle \om,\eta \rangle.
\end{equation}
 If  $m=1$, set $P(\om):=(x-1)(x-\al_1)$. We say that $\Omega$ is locally $($resp.\@ globally$)$ $s$-optimal if the polynomial 
$P$  is locally $($resp.\@ globally$)$ $s$-optimal in $\b^0_s(\sph{d-1})$.
\end{definition}

Since every tight spherical design generates a quadrature rule for the measure associated to Gegenbauer polynomials (see \S \ref{sec:quaddesigns}), the zonal symmetrization argument from the proof of Theorem \ref{thm:Sphere} leads to the following result.

\begin{proposition}\label{prop:lowerboundlocoptsphere}
Let $s\in\{+,-\}$ and $d\geq 2$. 
Let $\Omega\subseteq\sph{d-1}$ be a spherical $t$-design with $\al(\Omega)=\{\al_m<\al_{m-1}<\ldots<\al_1\}$. Let $f\in \b_s(\sph{d-1})\setminus\{{\bf 0}\}$ be a polynomial satisfying $\deg(f)\leq t$, and further assume  $f(\eta)=0$ if $s=+1$. Then $\theta(f)\geq \arccos(\al_1)$. Moreover, if $\theta(f)= \arccos(\al_1)$ and  $f$ is invariant under rotations that fix the north pole $\eta$, then $f$ coincides with a positive multiple of the polynomial $P$ defined in \eqref{eq:defPolyP}.
\end{proposition}

The discussion preceding Corollary \ref{cor:localoptdesignquad} below implies that every tight spherical $t$-design is in fact locally $s$-optimal. Moreover, in light of Proposition \ref{prop:lowerboundlocoptsphere}, a tight spherical $t$-design is globally $s$-optimal if the corresponding polynomial $P$ defined via  \eqref{eq:defPolyP} satisfies\footnote{Recall that $k(s\ft P)\geq 2$ since $P\in \b_s^{0}(\sph{d-1})$.} $k(s\ft P)=2$. 
In the following examples, given a certain  set of nodes $X=(x_m,x_{m-1},\ldots,x_0)$,  $W=(w_m,w_{m-1},\ldots,w_0)$ will be such that $\left\{\frac{w_j}{\sum_{i=0}^m w_i}\right\}_{j=0}^m$ is the set of weights of the quadrature rule associated with the nodes in $X$. 

\begin{example}[Simplex]\label{ex:simplex}
The regular simplex on $\sph{d-1}$ is a tight spherical $2$-design with $d+1$ vertices and one cosine distance, $-\frac{1}{d}$. It induces a quadrature rule of degree $t=2$ for the Gegenbauer measure $w_{\nu-\frac12,\nu-\frac12}$ $($see \eqref{eq:walphabeta} below$)$, $\nu=\frac{d}2-1$, with $X=\left(\frac{-1}{2\nu+2},1\right)$ and $W=(2\nu+2,1)$.
One easily checks that this quadrature rule  integrates all polynomials of degree at most 2 exactly, for all $\nu\geq 0$. 
Moreover, letting\footnote{The {\it modified Gegenbauer polynomials} are defined as $G_n^\nu(x):=\nu^{-1}{C_n^\nu(x)}$ for $\nu\geq 0$, with the understanding that $G^0_n(x)=\lim_{\nu\to 0^+}\nu^{-1}{C_n^\nu(x)}$.}
$$
P(x)=(x-1)\left(x+\frac{1}{2\nu+2}\right)=\frac{-(2\nu + 1)}{4\nu + 4}G_1^\nu(x)+\frac{1}{2\nu+2}G_2^\nu(x),
$$
we have that $k(\ft P)=2$.  
Hence  $P$ is a globally $+ 1$-optimal polynomial in $\b_+^0(I;\nu-\frac12,\nu-\frac12)$, and the regular simplex is a globally $+ 1$-optimal tight 2-design on $\sph{d-1}$.
\end{example}

\begin{example}[Cross-polytope]\label{ex:cross}
The cross-polytope
$\{\pm e_j\}_{j=1}^d$ on $\sph{d-1}$ is a tight spherical $3$-design with $2d$ vertices and two cosine distances, $\{-1,0\}$. It induces a quadrature rule of degree $t=3$ for $w_{\nu-\frac12,\nu-\frac12}$, $\nu=\frac{d}2-1$, with
$X=(-1,0,1)$ and $W=(1,4\nu+2,1)$.
One easily checks that this quadrature rule integrates all polynomials of degree at most 3 exactly, for all $\nu\geq 0$. Moreover, letting
$$
P(x)=(x^2-1)x=\frac{-(2\nu + 1)}{4(\nu + 2)}G_1^\nu(x)+\frac{3}{4(\nu+1)(\nu+2)}G_3^\nu(x),
$$
we have that $k(\ft P)=2$. 
Hence  $P$ is a globally $+ 1$-optimal polynomial in $\b_+^0(I;\nu-\frac12,\nu-\frac12)$, and the cross-polytope is a globally $+ 1$-optimal tight 3-design on $\sph{d-1}$.
\end{example}
We summarize the preceding discussion in the following result.

\begin{theorem}\label{thm:cubature}
Let $d\geq 2$.
Every tight spherical $t$-design is locally $s$-optimal, for any $s\in\{+,-\}$. Furthermore:
\begin{itemize}
\item The regular simplex on $\sph{d-1}$ with $d+1$ vertices is a globally $+1$-optimal tight $2$-design$;$
\item  The cross-polytope on $\sph{d-1}$ with $2d$ vertices is a globally $+ 1$-optimal tight $3$-design.
\end{itemize}
\end{theorem}

We have not been able to find any globally $-1$-optimal design, nor any further globally $+1$-optimal designs. 

\subsection{Jacobi Polynomials}\label{sec:Jac}
Let $\{P_n^{(\alpha,\beta)}\}_{n\in\N}$ denote the Jacobi polynomials with parameters $\alpha,\beta>-1$.
These are defined in \cite[Ch.\@ IV]{Sz75} as the orthogonal polynomials on the interval $I:=[-1,1]$, associated with the measure 
\begin{equation}\label{eq:walphabeta}
w_{\alpha,\beta}(x)\d x= c_{\al,\beta}(1-x)^\alpha (1+x)^\beta \mathbbm 1_{I}(x)\d x,
\end{equation}
and normalized in such a way that
\begin{equation}\label{eq:JacobiNorm}
P_n^{(\alpha,\beta)}(1)=\binom{n+\alpha}{n}.
\end{equation}
If $\alpha=\beta=\nu-\tfrac12$, then 
\begin{align}\label{eq:JacobivsGegen}
P_n^{(\nu-\tfrac12,\nu-\tfrac12)}(x) = \frac{\binom{n+\nu}{n}}{\binom{n+2\nu-1}n}C_n^{\nu}(x),
\end{align}
where $C_n^{\nu}$ is the Gegenbauer polynomial of degree $n$ and order $\nu$. The constant $c_{\al,\beta}$ in \eqref{eq:walphabeta} is chosen in such a way that $w_{\alpha,\beta}(x)\d x$ defines a probability measure,
\begin{equation}\label{eq:Defcab}
c_{\alpha,\beta}^{-1}
=\int_{-1}^1 (1-x)^\alpha (1+x)^\beta\d x
=2^{\alpha+\beta+1}\frac{\Gamma(\alpha+1)\Gamma(\beta+1)}{\Gamma(\alpha+\beta+2)}.
\end{equation}
Rodrigues' formula \cite[(4.3.1)]{Sz75} states that
\[(1-x)^\alpha(1+x)^\beta P_n^{(\alpha,\beta)}(x)=\frac{(-1)^n}{2^n n!} \left(\frac{\d}{\d x}\right)^n[(1-x)^{n+\alpha}(1+x)^{n+\beta}],\]
from which it can be deduced that
\begin{align}\label{eq:hab}
\begin{split}
h_n^{(\alpha,\beta)}
&:=\int_{-1}^1 P_n^{(\alpha,\beta)}(x)^2 w_{\alpha,\beta}(x)\d x\\
&=\frac{1}{2n+\alpha+\beta+1}\frac{\Gamma(\alpha+\beta+2)\Gamma(n+\alpha+1)\Gamma(n+\beta+1)}{\Gamma(\alpha+1)\Gamma(\beta+1)\Gamma(n+1)\Gamma(n+\alpha+\beta+1)}.
\end{split}
\end{align}
Here, $(2n+\alpha+\beta+1)\Gamma(n+\alpha+\beta+1)$ has to be replaced by $\Gamma(n+\alpha+\beta+2)$ if $n=0$; see \cite[(4.3.3)]{Sz75}.
Setting
\[p_n^{(\alpha,\beta)}:=(h_n^{(\alpha,\beta)})^{-\frac12}P_n^{(\alpha,\beta)},\]
we then have that $\{p_n^{(\alpha,\beta)}\}_{n\in\N}$ constitutes an orthonormal basis of $L^2(I)=L^2(I,w_{\al,\beta})$.
Any real-valued function $f:[-1,1]\to\R$ in $L^2(I)$ can be decomposed as  
\begin{equation}\label{eq:JacobiSeries}
f(x)=\sum_{n=0}^\infty \widehat f(n){p_n^{(\alpha,\beta)}(x)},
\end{equation}
where $\widehat f(n)$ denotes the $n$-th coefficient of $f$ with respect to the orthonormal basis $\{{p_n^{(\alpha,\beta)}}\}_{n\in\N}$.

\begin{definition}[The $\mathcal{B}_s(I;\alpha,\beta)$-cone]\label{def:AsJac} 
Let $s\in\{+,-\}$, and let $\alpha\geq\beta\geq-\tfrac12$.
Then $\mathcal{B}_s(I;\alpha,\beta)$ consists of all continuous functions $f:[-1,1]\to\R$, such that:
\begin{itemize}
\item $\widehat f(0)\leq 0;$
\item $\{s\widehat f(n)\}_{n\in\N}$ is eventually nonnegative while  $sf(1)\leq 0$.
\end{itemize}
\end{definition}

The proof of Theorem \ref{thm:Jacobi} below will reveal that the space\footnote{Here, $d:I\times I\to[0,2]$ denotes the restriction of the usual Euclidean distance.} $(I,d,w_{\al,\beta}(x)\d x)$ is admissible in the sense of Definition \ref{def:AdmSp}, with respect to the basis $\{p_n^{(\alpha,\beta)}\}_{n\in\N}$ and $\frak 0=1$.
Moreover, $\mathcal B_s(I;\alpha,\beta)=\mathcal A_s(I)$ (recall Definition \ref{def:AsCone}).
Specializing \eqref{eq:defrf}, \eqref{eq:defkfhat} to the present case, we are led to consider
\begin{align}\label{def:kandrJac}
\begin{split}
r(f;I)&=\inf\{r\in(0,2]: f(x)\geq 0 \text{ if } x\in [-1,1-r)\};\\
k(s\widehat f)&=\min\{k\geq 1: s\widehat f(n)\geq 0 \text{ if } n\geq k\},
\end{split}
\end{align}
together with the quantity
\begin{equation}\label{eq:BsJac}
{\mathbb B}_s(I;\alpha,\beta):=\inf_{f\in \mathcal B_s(I;\alpha,\beta)\setminus\{{\bf 0}\}} r(f;I) k(s\widehat f)^2,
\end{equation}
 which is estimated by our next result.
 
\begin{theorem}\label{thm:Jacobi}
Let $s\in\{+,-\}$ and $\alpha\geq\beta\geq-\tfrac12$.
Then the following estimate holds:
\begin{equation}\label{eq:2sidedJac}
{\mathbb B}_s(I;\alpha,\beta)\geq \frac{2\Gamma(\alpha+2)^{\frac2{\alpha+1}}}{(4e^{\frac1{12}})^{\frac2{\alpha+1}} (\alpha+\beta+2)(\alpha+2)}.
\end{equation}
Moreover, $\B_+(I;\al,\be)\leq 2$ and $\B_-(I;\al,\be)\leq 8$.
\end{theorem}

\noindent{\it Remark.}
By Stirling's formula, the right-hand side of \eqref{eq:2sidedJac} satisfies
\[\frac{2\Gamma(\alpha+2)^{\frac2{\alpha+1}}}{(4e^{\frac1{12}})^{\frac2{\alpha+1}} (\alpha+\beta+2)(\alpha+2)}
=\frac{2e^{-2}}{1+\frac{\beta}{\alpha}}\left(1+O\left(\frac{\log(\alpha+2)}{\alpha+1}\right)\right).\]
The upper bounds $\B_+(I;\al,\be)\leq 2$ and $\B_-(I;\al,\be)\leq 8$ are produced by the polynomials 
\begin{equation}\label{eq:f+f-}
f_+(x)=-1+\frac{P_1^{(\al,\be)}(x)}{P_1^{(\al,\be)}(1)} \ \text{ and } \ 
f_-(x)=-\frac{P_1^{(\al,\be)}(x)}{P_1^{(\al,\be)}(1)}+\frac{P_2^{(\al,\be)}(x)}{P_2^{(\al,\be)}(1)},
\end{equation}
respectively. We have performed extensive numerical searches in order to find polynomials up to degree  $30$ which lead to better upper bounds, but were unable to find any. 
Nevertheless, we would be extremely surprised if the polynomials $f_{\pm}$ from \eqref{eq:f+f-} turned out to be extremal.\\

We are interested  in the following restricted optimum:
\begin{align*}
{\mathbb B}_s^0\left(I;\alpha,\beta\right):=\inf\bigg\{ r(f;I) k(s\widehat f)^2: \,  f\in \mathcal{B}_s\left(I;\alpha,\beta\right)\setminus \{{\bf 0}\},  f(1)=0\bigg\},
\end{align*}
which according to the next result coincides with \eqref{eq:BsJac}.

\begin{proposition}\label{prop:WLOG} 
Let $s\in \{+,-\}$, $\alpha\geq \beta\geq -\tfrac12$,  and $f\in \b_s(I;\al,\beta)\setminus \{{\bf 0}\}$. Then there exists a polynomial $g$ such that $f+g\in \b_s(I;\al,\beta)\setminus \{{\bf 0}\}$, $(f+g)(1)=0$, $k(s\ft f+s\ft g)=k(s\ft f)$, and $r(f+g;I) < r(f;I)$.
In particular, 
${\mathbb B}_s^0\left(I;\alpha,\beta\right) = {\mathbb B}_s\left(I;\alpha,\beta\right)$.
\end{proposition}

\subsubsection{\underline{Connections with Quadrature}}\label{sec:quadjacobi}
A finite set $\{(x_j,\la_j)\}_{j=0}^m$ with $-1\leq x_m<x_{m-1}<\ldots<x_0\leq 1$ and $\la_j>0$ for $j=0,\ldots,m$ is said to generate a quadrature rule of degree $t$ for the measure $w_{\al,\beta}$ if, for every polynomial $f$ of degree at most $t$,  
$$
\int_{-1}^1 f(x)w_{\al,\beta}(x)\d x= \sum_{j=0}^m \la_j f(x_j).
$$
$X:=\{x_j\}_{j=0}^m$ is the set of {\it nodes} and $\Lambda:=\{\la_j\}_{j=0}^m$ is the set of {\it weights}. Note that necessarily $t\leq 2m+1$, for otherwise the integral of the polynomial $\prod_{j=0}^m(x-x_j)^2$ against the measure $w_{\al,\beta}$ would be zero, which is absurd. Similarly, if $x_m=-1<-x_0$ or $x_m>-1=-x_0$, then $t\leq 2m$, and if $x_0=-x_m=1$, then $t\leq 2m-1$. 

Quadrature rules where $t$ is as large as possible can be completely classified via the Gauss--Jacobi quadrature \cite[Theorem 3.4.1]{Sz75}, with nodes given by the zeros of Jacobi polynomials, and weights given by the Christoffel numbers; see \cite{DGS77}. A quick review follows.

\begin{itemize}
\item Assume that $-1<x_m<x_0<1$ and $t=2m+1$. Then $q(x)=\prod_{j=0}^m(x-x_j)$ is orthogonal to all polynomials of degree $\leq m$ with respect to the measure $w_{\al,\be}$, and therefore $q=c\, p_{m+1}^{(\al,\be)}$, for some $c>0$.

\item Assume that $-1=x_m<x_0<1$ (resp. $-1<x_m<x_0=1$) and $t=2m$. Then $q(x)=\prod_{j=0}^{m-1}(x-x_j)$ (resp. $q(x)=\prod_{j=1}^m(x-x_j)$) is orthogonal to all polynomials of degree $\leq m-1$ with respect to $w_{\al,\be+1}$ (resp. $w_{\al+1,\be}$), and therefore $q=c\, p_{m}^{(\al,\be+1)}$ (resp. $q=c\, p_{m}^{(\al+1,\be)}$), for some $c>0$.

\item Assume that $-1=x_m<x_0=1$ and $t=2m-1$. Then $q(x)=\prod_{j=1}^{m-1}(x-x_j)$ is orthogonal to all polynomials of degree $\leq m-2$ with respect to $w_{\al+1,\be+1}$, and therefore $q=c\, p_{m-1}^{(\al+1,\be+1)}$, for some $c>0$.
\end{itemize}

\begin{definition}[$s$-optimal polynomial in $\b_s^0(I;\al,\beta)$]\label{def:localoptgegen}
Let $s\in\{+,-\}$ and $\alpha\geq\beta\geq-\frac12$.
A polynomial $f\in \b_s^0(I;\al,\beta)$ is locally $s$-optimal if there exists $\delta>0$, such that 
\[r(f;I) k(s\ft f)^2< r(h;I)k(s\ft h)^2,\]
for any polynomial $h\in\b_s^0(I;\al,\beta)$ satisfying $\deg(h)\leq \deg(f)$ and $0<\inf_{c>0}\|f-ch\|_{L^\infty(I)}<\delta$.
The polynomial $f$ is said to be globally $s$-optimal if one can take $\delta=+\infty$.
\end{definition}

In what follows, we let $x_{1,m}^{(\al,\be)}$ denote the largest zero of the polynomial $p_{m}^{(\al,\be)}$.

\begin{theorem}\label{thm:localoptuadrature}
Let $\alpha\geq\beta\geq-\frac12$. Define the polynomials
\begin{align}\label{eq:localoptpolyquadrature}
\begin{split}
P(x) & := (1-x)\frac{p_{m}^{(\al+1,\be)}(x)^2}{x_{1,m}^{(\al+1,\be)}-x}, \ \ \ (m\geq 1); \\ 
Q(x) & := (1-x^2)\frac{p_{m-1}^{(\al+1,\be+1)}(x)^2}{x_{1,m-1}^{(\al+1,\be+1)}-x}, \ \ \ (m\geq 2).
\end{split}
\end{align} 
Then $P$ and $Q$ are locally $s$-optimal in $\b^0_s(I;\al,\beta)$, for any $s\in\{+,-\}$.
\end{theorem}

\subsubsection{\underline{Quadrature and Spherical Designs}}\label{sec:quaddesigns}
Aiming to establish a connection between spherical designs and the sign uncertainty principle for spherical harmonics,
we now restrict attention to Gegenbauer polynomials.
For notational simplicity, set $\mu_\nu:= w_{\nu-\frac12,\nu-\frac12}$. 
Let $\Omega\subseteq \sph{d-1}$ be a tight spherical $t$-design with set of cosine distances $\{\al_m<\al_{m-1}<\ldots<\al_1\}\subseteq [-1,1)$, where $t=2m$ if $\al_m>-1$, and $t=2m-1$ if $\al_m=-1$. Define 
$$
\ell_j:=\#\{(\om,\om')\in \Omega^2: \langle \om,\om'\rangle=\al_j \},
$$
and further set $\ell_0=1$,  $x_0=1$, and $\{x_{j}=\al_j\}_{j=1}^m$. 
We note that $\{(x_j,\frac{\ell_j}{\# \Omega^2})\}_{j=0}^m$ generates a quadrature rule of degree $t$ for $\mu_\nu$. Indeed, if $f$ is a polynomial of degree at most $t$, and $\bar\sigma$ denotes the normalized surface measure on $\sph{d-1}$, then
$$
\int_{(\sph{d-1})^2} f(\langle \zeta,\nu \rangle) \d \bar \sigma(\zeta) \d \bar \sigma(\nu)  =  \frac{1}{\#\Omega^2}\sum_{\om,\om'\in \Omega} f(\langle \om,\om' \rangle) = \sum_{j=0}^m \frac{\ell_j}{\#\Omega^2} f(x_j).
$$
In particular, the sequence $\{\al_j\}_{j=1}^m\setminus \{- 1\}$ coincide with the zeros of the polynomial $p_{m}^{(\nu+1/2,\nu-1/2)}$ or $p_{m-1}^{(\nu+1/2,\nu+1/2)}$, depending on whether $\al_m>-1$ or $\al_m=-1$, respectively. On the other hand, if $\eta\in\sph{d-1}$ denotes the north pole as usual, then
$$
\int_{(\sph{d-1})^2}  f(\langle \zeta,\nu \rangle) \d \bar \sigma(\zeta) \d \bar \sigma(\nu) = \int_{\sph{d-1}} f(\langle \zeta,\eta\rangle) \d \bar \sigma(\zeta) = \int_{-1}^1 f(x)\mu_\nu(x)\d x.
$$
Moreover, it is straightforward to verify that the map $f(x) \mapsto F(\om):=f(\langle \om,\eta\rangle)$ defines a bijection between the sets $\b_s^0(I;\nu-\frac12,\nu-\frac12)$ and  $\b_s^0(\sph{d-1})$, and that $k(s\ft f)=k(s\ft F)$ and $r(f;I)=1-\cos(\theta(F))$.
With these considerations in place, one easily checks that Theorem \ref{thm:localoptuadrature} specializes to the following result.

\begin{corollary} \label{cor:localoptdesignquad}
Let $d\geq 2$, and set $\al=\beta=\tfrac{d-3}{2}$
in Theorem \ref{thm:localoptuadrature}. Then, for any $s\in\{+,-\}$, the polynomials $f:=P(\langle \cdot,\eta\rangle )$ and $g:=Q(\langle \cdot,\eta\rangle )$ $($where $P,Q$ were defined in \eqref{eq:localoptpolyquadrature}$)$ are locally $s$-optimal in $\b_s^0(\sph{d-1})$ .
\end{corollary}


\subsection{Fourier Series}\label{sec:Tor}
Given $d\geq 1$, the $d$-torus $\T^d=\R^d/\Z^d$ can be defined as the set of equivalence classes under the equivalence relation $x\sim y$ if $x-y\in\Z^d$.
Equivalently, we will think of $\T^d$ as the following subset of $\Co^d$:
\[\T^d=\{(e^{2\pi i x_1},\ldots,e^{2\pi i x_d})\in\Co^d: (x_1,\ldots,x_d)\in[-\tfrac12,\tfrac12]^d\}\]
Functions on $\T^d$ are in one-to-one correspondence with functions on $\R^d$ which are 1-periodic in each coordinate.
The Haar probability measure on $\T^d$, denoted $\lambda$, is simply the restriction of $d$-dimensional Lebesgue measure to the unit cube $[-\frac12,\frac12]^d$.
Translation invariance of the Lebesgue measure, and periodicity of functions on $\T^d$, imply that
\[\int_{\T^d} f\d \lambda=\int_{[-\frac12,\frac12]^d} f(x)\d x.\]
Given a real-valued function $f\in L^1(\T^d)=L^1(\T^d,\lambda)$, and $m\in\Z^d$, 
define the corresponding Fourier coefficient
\[\widehat{f}(m)=\int_{\T^d} f(x) e^{-2\pi i \langle x,m\rangle} \d\lambda(x).\]
An immediate consequence is the estimate 
$\|\widehat{f}\|_{\ell^\infty(\Z^d)}\leq \|f\|_{L^1(\T^d)}.$
If $f\in L^1(\T^d)$ and $\widehat{f}\in\ell^1(\Z^d)$, then Fourier inversion applies, and implies that, for $\lambda$-almost every $x\in\T^d$,
\[f(x)=\sum_{m\in\Z^d} \widehat{f}(m) e^{2\pi i \langle x,m\rangle}.\]
In particular, $f$ is almost everywhere equal to a continuous function on $\T^d$; see \cite[Prop.\@ 3.1.14]{Gr08}.
If moreover $f\in L^2(\T^d)$, then Plancherel's identity states that
\[\|f\|_{L^2(\T^d)}^2=\sum_{m\in\Z^d} |\widehat{f}(m)|^2.\]
As an immediate consequence of Theorem \ref{thm:OSUP}, we obtain the following result. 
\begin{theorem}\label{thm:torus}
Let $s\in\{+,-\}$, $d\geq 1$.
Let $f\in L^1(\T^d)$ be nonzero and such that $\widehat{f}\in \ell^1(\Z^d)$,  
\[\int_{\T^d} f\d\lambda\leq 0,\text{ and }\sum_{m\in\Z^d} s\widehat{f}(m)\leq 0.\]
Then the following inequality holds:
\[\lambda(\{x\in\T^d: f(x)<0\})\cdot\#\{m\in\Z^d:s\widehat{f}(m)<0\}\geq \frac1{16}.\]
\end{theorem}
 The space $(\T^d,d_\infty,\lambda)$ is admissible for $\frak 0=(0,\ldots,0)\in\T^d$ in the sense of Definition \ref{def:AdmSp}.
Here, $d_\infty:\T^d\times\T^d\to[0,1]$ is defined via
\[d_\infty(x,y):=\max_{1\leq j\leq d} |x_j-y_j|,\]
where $|x|$ denotes the distance from $x$ to $0$ in $\T^1$.
The following result then follows from Theorem \ref{thm:MSUP}, or more directly from Theorem \ref{thm:torus}.

\begin{theorem}\label{thm:Torus}
Let $s\in\{+,-\}$, $d\geq 1$.
Let $f\in\mathcal A_s(\T^d)$ be a nonzero, even function, for which there exist $r_f\in (0,1],k_{s\ft f}\geq 1$ with the following properties:
$f(x)\geq 0$ if $d_\infty(x,\mathfrak 0)\geq r_f$ while $\widehat f(0)\leq 0$, and
$s\ft f(m)\geq 0$ if $|m|\geq k_{s\ft f}$ while $sf(\mathfrak 0)\leq 0$.
Then the following inequality holds:
\begin{equation}\label{eq:MSUPTd}
r_f (2k_{s\ft f}-1)\geq 2^{-(1+\frac 4d)}.
\end{equation}
\end{theorem}

In the companion paper \cite{GOSR20}, we established the following estimate:
\begin{equation}\label{eq:TransPrinc}
\inf_{f\in \a_+(\T^1)\setminus\{{\bf 0}\}}{\sqrt{r_f k_{\ft f}}} \leq \A_+(1);
\end{equation}
see \cite[Prop.\@ 4]{GOSR20}.
We do not know whether an analogous result holds for $s=-1$.
 Another open problem is to determine whether  equality holds in \eqref{eq:TransPrinc}, in which case the statement could be regarded as a transference principle between the continuous and discrete settings. 
 It would also be interesting to prove a similar result for Dini series, which should relate to the higher dimensional $\pm 1$ uncertainty principles $\A_s(d)$, $d\geq 2$, and are the subject of the next section.

\subsection{Dini Series}\label{sec:diniseries}
The Dini series of a function $f:[0,1]\to \R$  is given by
\begin{equation}\label{eq:Dini}
f(x)=\frak B_0(x)+\sum_{n=1}^\infty c_n J_\nu(\lambda_n x),
\end{equation}
where $0<\lambda_1<\lambda_2<\ldots$ denote the positive zeros 
of the function
\begin{equation}\label{diniserisfunction}
zJ_\nu'(z)+HJ_\nu(z) = (H+\nu)J_\nu(z)-zJ_{\nu+1}(z).
\end{equation}
Here, $J_\nu$ is the Bessel function of the first kind of order $\nu\geq -\frac12$, and $H\in\R$. The initial term in \eqref{eq:Dini}, $\frak B_0(x)$, depends on the sign of $H+\nu$. If $H+\nu>0$, then $\frak B_0\equiv 0$; if $H+\nu<0$, then the function \eqref{diniserisfunction} has two purely imaginary zeros $\pm i\lambda_0$, whose contributions are manifested by taking $\frak B_0(x)$ to be an appropriate multiple of $J_\nu(i\lambda_0 x)$; if $H+\nu=0$, then the imaginary zeros coalesce at the origin, and $\frak B_0(x)=2(\nu+1)x^{\nu}\int_0^1 t^{\nu+1}f(t)\d t$. Note that the functions $x\mapsto J_\nu(\lambda_n x), n\in\N,$ are orthogonal in $[0,1]$ with respect to the measure $x\d x$. Indeed, \cite[\S 5.11-(8)]{W66} implies that, for all real numbers $k\neq \ell$,
\begin{equation}\label{besselidentity}
\int_0^1 J_\nu(k x)J_\nu(\ell x)x\d x = \frac{kJ_{\nu+1}(k)J_{\nu}(\ell)-\ell J_{\nu}(k)J_{\nu+1}(\ell)}{k^2-\ell^2}.
\end{equation}
 If $k,\ell$ are distinct zeros of \eqref{diniserisfunction}, then one can invoke the usual recurrence relations for Bessel functions in order to deduce that the integral in \eqref{besselidentity} vanishes.  

If $H+\nu=0$, then the elements of the sequence $\{\la_n\}_{n\geq 1}$ coincide with the positive zeros of the function $J_{\nu+1}$. In this case, if  $\nu=-\frac12$, then $J_{\nu+1}(x) = {(\tfrac{2}{\pi x})^{\frac12}}\sin(x)$ and $\la_n=\pi n$; hence the Dini series \eqref{eq:Dini} specializes to the Fourier  series from \S \ref{sec:Tor}.  In this way, Dini series for $H+\nu=0$ are seen to generalize  one-dimensional Fourier series to the higher dimensional radial case.

In order to properly place Dini series within the scope of Theorem \ref{thm:MSUP}, we need to normalize the functions  $J_\nu(\lambda_n x)$, in such a way as to ensure that their maximum is attained at the origin. This is most easily done by introducing the even, entire function $A_\nu(z):=\Gamma(\nu+1)(\frac12 z)^{-\nu} J_\nu(z)$, since $|A_\nu(z)|\leq A_\nu(0)=1$. 
One can then rescale the results from \cite[\S 18.33]{W66}, and invoke the identity \cite[\S 5.11-(11)]{W66},
$\int_0^1 A_\nu^2(\lambda_n x) {x^{2\nu+1}} \d x= \frac{A_\nu^2(\lambda_n)}2,$ in order to derive the following proposition.

\begin{proposition}\label{prop:diniseries}
Let $\nu\geq -\frac12$. For every $f\in L^2\left([0,1], \frac{x^{2\nu+1}}{2(\nu+1)} \d x \right)$, we have that
\begin{equation}\label{eq:DiniSeries}
f(x)=\ft f(0) + 2\sqrt{\nu+1} \sum_{n=1}^\infty\ft f(n)  \frac{A_\nu(\lambda_n x )}{A_\nu(\lambda_n)}
\end{equation}
in the $L^2$-sense, where $\{\lambda_{n}\}_{n\geq 1}$ denote the positive zeros of the Bessel function $J_{\nu+1}$,
\begin{equation}\label{eq:DiniCoeff}
\ft f(n)= \frac{2\sqrt{\nu+1}}{A_\nu(\lambda_n) }\int_0^1 f(x)A_\nu(\lambda_{n} x) \frac{x^{2\nu+1}\d x}{2(\nu+1)},
\end{equation}
for all $n\geq 1$, and
$\ft f(0)=\int_0^1 f(x) \frac{x^{2\nu+1}\d x}{2(\nu+1)}$.
Moreover, if $f$ is continuous and of bounded variation in $[0,1]$, then the Dini series \eqref{eq:DiniSeries} of $f$ converges absolutely and uniformly in $[0,1]$.
\end{proposition}

Identity \cite[\S12.11-(1)]{W66} translates into
$\int_0^1 A_\nu(k x)x^{2\nu+1}\d x = \frac{A_{\nu+1}(k)}{2(k+1)}$,
and reveals that the functions $\{A_\nu(\la_n x)\}_{n\geq 1}$ are orthogonal to the constant function ${\bf 1}$. Consequently, the orthonormal basis 
$$
\{{\bf 1}\}\cup\left\{\frac{2\sqrt{\nu+1}}{A_\nu(\lambda_n) } A_\nu(\lambda_{n} x)\right\}_{n\geq 1}
$$
satisfies all the hypotheses of Theorem \ref{thm:MSUP} with $\frak 0=0$. We can then use the well-known asymptotic formulae $\lambda_n\sim \pi n$ and 
$$
J_\nu(z)=\sqrt{\frac{2}{\pi z}}\cos(z-\nu\pi/2-\pi/4) + O(|z|^{-3/2}),
$$
see \cite[\S 7.1]{W66}, in order to deduce that
$A_\nu(\lambda_n)^{-2} \sim \la_n^{2\nu+1},$
where the implied constant depends only on $\nu$. 
The following result  can then be derived from Theorem \ref{thm:MSUP} at once.

\begin{theorem}\label{thm:dini}
Let $s\in\{+,-\}$, $\nu\geq -\frac{1}{2}$. 
Let $f:[0,1]\to \R$ be a nonzero continuous function of bounded variation, whose coefficients $\{\ft f(n)\}_{n\geq 1}$ defined in \eqref{eq:DiniCoeff} satisfy
$$
\sum_{n=1}^\infty n^{\nu+\tfrac12}{|\ft f(n)|} <\infty.
$$
Suppose that there exist $r_f\in(0,1]$, $k_{sf}\geq 1$, such that $f(x)\geq 0$ if $x\in[r_f,1]$ while $\ft f(0)\leq 0$, and $s \ft f(n)\geq 0$ if $n\geq k_{sf}$ while $s f(0)\leq 0$. Then there exists $c_\nu>0$, such that
\begin{equation}\label{eq:DiniResult}
r_f \, k_{sf}^{2\nu+2} \geq c_\nu.
\end{equation}
\end{theorem}
The constant $c_\nu$ in \eqref{eq:DiniResult} depends only on $\nu$ and can be made explicit, e.g.\@ by appealing to \cite[Lemma 2.5]{OST17}. However, the number of terms in the required asymptotic expansion grows linearly with the parameter $\nu$, and as such we have omitted the precise formulation of the corresponding (somewhat cumbersome) statement.

\section{Sign Uncertainty in Discrete Spaces}\label{sec:discspaces}

\subsection{Discrete Fourier Transform}\label{sec:DFT}
Let $q\geq 1$ be an integer, and let $\Z_{2q+1}$ denote the set of equivalence classes of integers modulo $2q+1$. The choice of a residue class of odd size is convenient\footnote{On the other hand, everything that follows can be easily adapted to residue classes of arbitrary size.} for numerical purposes, since we can then place the origin (in the sense of Definition \ref{def:AdmSp}) at $n=0$.

If $f:\Z_{2q+1}\to\R$ is real-valued and even, then  its discrete Fourier transform $\ft f$, defined via
\begin{equation}\label{eq:FCoef}
\widehat{f}(k)=\frac1{\sqrt{2q+1}}\sum_{n=-q}^{q} f(n) e^{- 2\pi i\frac { k n} {2q+1}} =\frac{1}{\sqrt{2q+1}}\left(f(0) + 2\sum_{n=1}^q f(n)\cos\left(2\pi \frac{kn}{2q+1}\right)\right)
\end{equation}
is likewise real-valued and even. Since the discrete Fourier transform defines an isometry from $L^2(\Z_{2q+1})\simeq\R^{2q+1}$ onto itself, and $\max_{-q\leq k\leq q} |\ft f(k)| \leq (2q+1)^{-\frac12}\sum_{n=-q}^q |f(n)|$, the following result is a direct consequence of Theorem \ref{thm:OSUP}. 

\begin{theorem}\label{thm:UPZq}
Let $s\in\{+,-\}$ and $q\geq 1$ be an integer.
Let $f:\Z_{2q+1}\to\R$ be nonzero and even.
Assume that $sf(0)\leq 0$ and $\widehat{f}(0)\leq 0$.
Then the following inequality holds:
\begin{equation}\label{eq:UPZq}
\#\{n\in\Z_{2q+1}:f(n)<0\}\cdot\#\{k\in\Z_{2q+1}:s\widehat{f}(k)<0\}\geq\frac {2q+1}{16}.
\end{equation}
\end{theorem}

The following problem will be of interest.
\begin{prob}[Feasibility Linear Programming Problem for the discrete Fourier transform]\label{DiscFeasLP}
Given $s\in\{+,-\}$, let
\begin{equation}\label{discLP}
\A_s^{\textup{disc}}(q):=\min\{k_{sf}\geq 0: f\in \a_s^{\textup{disc}}(q)\},
\end{equation}
where $\a_s^{\textup{disc}}(q)$ denotes the set of even functions $f:\Z_{2q+1}\to\R$, such that $sf(0),\ft f(0)\leq 0$ and $f(\pm q),s\ft f(\pm q)\geq 1$, and $k_{sf}$ is the smallest nonnegative integer, for which $f(n),s\ft f(n)\geq 0$ if $k_{sf}\leq |n|\leq q$. Here, $|n|$ denotes the absolute value of the representation of $n$ in the interval $\{-q,\ldots,0,\ldots,q\}$.
\end{prob}

\begin{definition}[$s$-Feasibility]\label{def:sfeasibility}
Let $s\in\{+,-\}$. A pair $(k,q)$ is $s$-feasible if there exists $f\in \a_s^{\textup{disc}}(q)$, such that $k_{sf}\leq k$.
\end{definition}

The following result is an immediate consequence of Theorem \ref{thm:UPZq} and Definition \ref{discLP}.
\begin{corollary} Let $s\in\{+,-\}$ and $q\geq 1$ be an integer. Then
$$\frac{\A_s^{\textup{disc}}(q)}{\sqrt{2q+1}} \geq \frac{1}{8}.$$
\end{corollary}

 Problem \ref{DiscFeasLP} can be solved numerically with a linear programming solver, and we have done so.  Numerical evidence presented in \S \ref{numerics:DFT} strongly supports the following conjecture. 

\begin{conjecture}\label{conj:discFTpropertyconj} Let $s\in\{+,-\}$. If $(k,q)$ is $s$-feasible, then $(k+1,q),(k,q-1)$ are $s$-feasible. The function $q\mapsto \A_s^{\textup{disc}}(q)$ is non-decreasing,  and  its range contains all integers $k\geq 2$ if $s=+1$, and all integers $k\geq 3$ if $s=-1$. Moreover, 
$$
\lim_{q\to\infty} \frac{\A^{\textup{disc}}_s(q)}{\sqrt{2q+1}} = \A_s(1).
$$
where $\A_{s}(1)$ denotes the optimal constant for the one-dimensional continuous sign uncertainty principles defined in \eqref{defAd+}, \eqref{defAd-}. 

\end{conjecture}

Since the discrete Fourier transform is a proper discretization of the Fourier transform \eqref{def:fouriertrans}, it is natural to expect that the discrete uncertainty principles converge to their continuous counterparts, in the limit when $q\to\infty$. Indeed, this is what seems to happen numerically. Moreover,  the patterns in \S \ref{numerics:DFT} (see Table \ref{table:1}) are relatively straightforward to identify, and they provide evidence towards the following conjecture.

\begin{conjecture}\label{conj:discFTfeaspairs}
The pair  $(k,\left\lceil \frac{(k-1)^2}{2}\right\rceil)$ is $-1$-feasible, for every integer $k\geq 4$. Moreover, if 
$\widetilde q_-(k)=\left\lceil\frac{(k-1)^2}{2}\right\rceil$, then $k=\A^{\rm disc}_s(\widetilde q_s(k))+o(k)$.
\end{conjecture}

In this way, Conjectures \ref{conj:discFTpropertyconj} and \ref{conj:discFTfeaspairs} together imply  $\A_-(1) = 1$, which is known to hold; recall the discussion in \S \ref{moreback}, and see \S \ref{numerics:DFT} below for further details. 

We have performed extensive numerical computations for Problem \ref{DiscFeasLP} using the {\it Gurobi} linear programming solver \cite{gurobi} implemented via PARI/GP \cite{gp}, which we discuss in \S\ref{sec:numerics}. Numerically we observed the dimension of the cone of optimal vectors $f\in \a_s^{\rm disc}(q)$ for Problem \ref{DiscFeasLP} which satisfy $k_{sf}=\A_{s}^{\rm disc}(q)$ to be large. Further numerical experiments revealed that a good selection method consists in  finding an optimal vector $f\in \a_s^{\rm disc}(q)$ for which the corresponding {energy}, $\sum_{|n|\geq k_{sf}} f(n)^2$, is minimized. In particular, the plot of such a vector appears to be quite smooth.\footnote{Recall that the Gibbs phenomenon permeates throughout such numerical computations, and one should find ways to reduce it.} In the $-1$ case, we were able to exactly identify the vector $f_\star \in \A_-^{\rm disc}(q)$ delivered by the the solver after energy was minimized. We observed that
\begin{equation}\label{eq:fstarapprox}
f_\star(n)\approx \sin(2\pi |x|){\bf 1}_{[-1,1]}(x)-\frac{2\sin^2(\pi x)}{\pi(1-x^2)}
\end{equation}
for $x=n/\sqrt{2q+1}$ and $|n|\leq q$. Indeed, the function on the right-hand side of \eqref{eq:fstarapprox} is admissible and optimal for the continuous $-1$ uncertainty principle, revealing once again that $\A_{-}(1)=1$. Our next results makes these numerical observations precise, and adds weight to the validity of Conjecture \ref{conj:discFTpropertyconj}.

\begin{proposition}\label{prop:discguess-1}
Assume  $2q+1=\ell^2$, for some integer $\ell\geq 3$, and set
$$
g(n) =  \sin(2\pi |n| /\ell){\bf 1}_{[-\ell,\ell]}(n),
$$
so that, for $|n|\leq q$,  
$$
\ft g(n)=\frac{2 \sin^2(\pi n/\ell) \sin(2\pi/\ell)}{\ell(\cos(2\pi n/\ell^2)-\cos(2\pi/\ell))}.
$$
Let
$
f_\star = g-\ft g.
$
Then $f_\star\in \A_{-}^{\rm disc}(1)$, $\ft f_\star=- f_\star$, $f_\star(0)=0$, and $k_{-f_\star}=\ell$. Hence 
$$
\frac{\A_-^{\rm disc}(q)}{\sqrt{2q+1}} \leq 1.
$$
In general, if $2q+1$ is not a perfect square, then
$$
\frac{\A_-^{\rm disc}(q)}{\sqrt{2q+1}} \leq \sqrt{1+\frac{1+\sqrt{2q}}{2q+1}},
$$
for all $q\geq 5$. In particular, $\limsup_{q\to\infty} \frac{\A_-^{\rm disc}(q)}{\sqrt{2q+1}} \leq 1.$
\end{proposition}
\begin{proof}
Setting $x=n/\ell$, a straightforward computation shows that\footnote{Note that if $n=\ell$, then the numerator vanishes with the same order as the denominator.}
\begin{equation}\label{eq_sincos}
    \ft g(n)=\frac{2}{\ell} \sum_{j=1}^\ell \sin(2\pi j/\ell) \cos(2\pi j x/\ell) = \frac{2 \sin^2(\pi x) \sin(2\pi/\ell)}{\ell(\cos(2\pi x/\ell)-\cos(2\pi/\ell))}.
\end{equation}
To verify \eqref{eq_sincos}, replace sine and cosine by the corresponding exponential representations, note that the resulting sums are geometric and thus can be calculated explicitly, and rearrange terms. The claimed properties of the function $f_\star=g-\widehat g$ are easy to deduce, and we leave the details to the reader.
For any given $q\geq 5$ for which $2q+1$ is not a perfect square, we can simply take $\ell \geq 4$ such that $(\ell-1)^2 < 2q+1 \leq \ell^2$; in particular, $q\geq \ell$. Then $g$ can be seen as a vector in $\a^{\rm disc}_-(q)$ and, by the same computations as above,
$\ft g(n)\leq 0$ if $|n|\geq \lceil (2q+1)/\ell \rceil$. We obtain
$$
\frac{\A_-^{\rm disc}(q)}{\sqrt{2q+1}} \leq \sqrt{\frac{\lceil (2q+1)/\ell \rceil \ell }{2q+1}} \leq \sqrt{1+\frac{\ell}{2q+1}}\leq \sqrt{1+\frac{1+\sqrt{2q}}{2q+1}},
$$
as desired. This concludes the proof of the proposition.
\end{proof}
For every fixed $x\in\R$, we have that 
$$
f_\star(\lfloor\ell x\rfloor) \to \sin(2\pi |x|){\bf 1}_{[-1,1]}(x)-\frac{2\sin^2(\pi x)}{\pi(1-x^2)}, \quad \text{ as } \ell\to \infty.
$$
Numerically we have confirmed that $\A_-^{\rm disc}((\ell^2-1)/2)=\ell$, for every $\ell\leq 100$. 
It would be nice to find a proof along the lines of the reasoning above, showing that $\A_-^{\rm disc}((\ell^2-1)/2) \geq \ell$.

\begin{conjecture}
$\A_-^{\rm disc}((\ell^2-1)/2) = \ell$, for every integer $\ell\geq 3$.
\end{conjecture}

\subsection{Discrete Hankel Transform}\label{sec:DHT}
The discrete Hankel transform was proposed by Siegman in 1977, and later on several other versions were put forward; see \cite{F87}.
To the best of our knowledge, none of the proposed explicit forms defines a unitary operator; rather, they are only {\it asymptotically unitary}. In one way or another, they all properly discretize a given compactly supported function $f$, and then appeal to Bessel--Fourier series  in order to further discretize the Hankel transform of $f$.  Fisk Johnson \cite{F87} proposes several approaches, which turn out to work well in practice since they are already very close to being unitary when applied to ``short'' vectors.  Since Theorem \ref{thm:OSUP} only requires approximate inversion, it seems reasonable to expect that a sign uncertainty principle holds for each of the kernels defined in \cite[(13) \& (16)--(19)]{F87}; for the sake of brevity, we chose not  to fully pursue this line of investigation. 

The main purpose of this section is to formulate a sign uncertainty principle for the discrete Hankel transform of Fisk Johnson, and to start discussing the numerical experiments which we conducted.  Since (after normalization) the Hankel transform of order $\nu=\tfrac{d}2-1$ coincides with the Fourier transform of a radial function in $\R^d$, one may expect that, in the limit, the corresponding discrete sign uncertainty principle converges to the continuous sign uncertainty principle in all dimensions. We proceed to describe the evidence we obtained in support of this possibility.

Given $\nu\geq -\tfrac12$, let $\{j_n\}_{n\geq 1}$ denote the positive zeros of the Bessel function $J_\nu$. 
Our starting point is formula \cite[(13)]{F87}, for $N=q+1$ and $T=\sqrt{j_{q+1}}$.
Fisk Johnson proposes a discretization of the following version of the Hankel transform of parameter $\nu\geq -\tfrac12$,
\begin{equation}\label{eq:HankelPreDisc}
\widetilde H_\nu(f)(x) = \int_0^\infty f(y)J_\nu(xy)y\d y,
\end{equation}
 which we proceed to describe. Define the discrete Hankel transform with parameter $\nu\geq -\tfrac12$ of a given\footnote{Here, $[q]:=\{1,2,\ldots,q\}$.} $f:[q]\to\R$, as follows:
$$
H^{\textup{disc}}_{\nu} (f)(m) = \frac{2}{j_{q+1}}\sum_{n=1}^q f(n)\frac{J_\nu(j_mj_n/j_{q+1})}{J_{\nu+1}(j_n)^2}.
$$
Each of the values $f(n)$ is to be interpreted as the evaluation of some continuous function at the node ${j_n}{(j_{q+1})}^{-\tfrac12}$.
By showing that the kernel of the composition $H^{\textup{disc}}_{\nu} H^{\textup{disc}}_{\nu}$ satisfies\footnote{Here, $\delta_{m,\ell}$ denotes the usual Kronecker delta: $\delta_{m,\ell}=1$ if $m=\ell$, and $\delta_{m,\ell}=0$ otherwise.}
$$
\frac{4}{J_{\nu+1}(j_\ell)j_{q+1}^2}\sum_{n=1}^{q}\frac{J_\nu(j_mj_n/j_{q+1})J_\nu(j_nj_\ell/j_{q+1})}{J_{\nu+1}(j_n)^2}={ \delta}_{m,\ell} + o(1), \text{ as } q\to\infty,
$$
 where the term $o(1)$ is already small for small values of $q$, the author argues that 
$
H^{\textup{disc}}_{\nu} H^{\textup{disc}}_{\nu}  \approx \text{Id}
$; see \cite[(11)]{F87}.
We turn to the following feasibility problem.
\begin{prob}[Feasibility Linear Programming Problem for the discrete Hankel transform]
Given $s\in\{+,-\}$, let
\label{DiscFeasLP2}
\begin{equation}\label{discLP2}
\A_s^{\textup{disc}}(q,\nu):=\min\{k_{sf}: f\in \a_s^{\textup{disc}}(q,\nu)\},
\end{equation}
where $\a_s^{\textup{disc}}(q,\nu)$ denotes the set of functions $f:[q]\to\R$, such that $sf(1), H^{\textup{disc}}_{\nu} (f)(1)\leq 0$ and $f(q), s\ft f(q)\geq 1$, and $k_{sf}$ is the smallest nonnegative integer for which $f(n), \, sH^{\textup{disc}}_{\nu} (f)(n)\geq 0$ if $k_{sf}\leq n\leq q$.
\end{prob}

\begin{definition}[$(s,\nu)$-Feasibility]\label{def:snufeasibility}
Let $s\in\{+,-\}, \nu\geq -\tfrac12$.
A pair $(k,q)$ is $(s,\nu)$-feasible it there exists $f\in \a_s^{\textup{disc}}(q,\nu)$, such that $k_{sf}\leq k$.
\end{definition}

 In \S \ref{numerics:DHT} below, we present compelling numerical evidence towards the following conjecture. 

\begin{conjecture}\label{conj:discHTpropertyconj}
Let $s\in\{+,-\}, \nu\geq -\tfrac12$.
 If $(k,q)$ is $(s,\nu)$-feasible, then $(k+1,q), (k,q-1)$ are $(s,\nu)$-feasible. The function $q\mapsto \A_s^{\textup{disc}}(q,\nu)$ is non-decreasing, and its range contains $\N\setminus[k_0]$, for some $k_0\geq 1$. Moreover, if $\nu=\tfrac{d}2-1$ and $n_q=\A_s^{\textup{disc}}(q,\nu)$, then
\begin{equation}\label{limitDHT}
\lim_{q\to\infty} \frac{j_{n_q}}{\sqrt{2\pi j_{q+1}}} = \A_s(d),
\end{equation}
where $\A_s(d)$ denotes the optimal constant for the continuous sign uncertainty principles defined in \eqref{defAd+}, \eqref{defAd-}, and $\{j_n\}_{n\geq 1}$ are the positive zeros of the Bessel function $J_\nu$. 
\end{conjecture}

If $f:\R^d\to\R$ is radial and $\nu=\tfrac{d}2-1$, then identity \eqref{hankelfourier} below can be rephrased as 
$$
|\xi|^{\tfrac{d}2-1}\ft f(\xi) = c_\nu\, \widetilde H_\nu[y^\nu f(y)](2\pi |\xi|),
$$
for some $c_\nu>0$, and therefore the factor $\sqrt{2\pi}$ in \eqref{limitDHT} is to be expected. The particular cases  $d\in\{8,12,24\}$ are especially interesting since it is known that $\A_-(8)=\A_+(12)=\sqrt{2}$ and $\A_-(24)=2$. In these cases, the numerical data presented in \S \ref{numerics:DHT} corroborate Conjecture \ref{conj:discHTpropertyconj}. Moreover, if $d\in\{2,8,12,24\}$, then our numerics point to the following more structured version of Conjecture \ref{conj:discHTpropertyconj}.

\begin{conjecture}\label{conj:discHTfeaspairs}
The following statements hold:
\begin{itemize}
\item  $\left(k,\lfloor\frac{\sqrt{3}(k^2-2k+2)}{4}\rfloor\right)$ is $(-1,\frac22-1)$-feasible, for every integer $k\geq 4;$
\item  $\left(k,\lfloor\frac{k^2}{4}\rfloor\right)$ is $(-1,\frac82-1)$-feasible, for every integer $k\geq 4;$ 
\item $\left(k,\lfloor\frac{k^2+6k-8}{8}\rfloor\right)$ is $(-1,\frac{24}2-1)$-feasible, for every integer $k\geq 4;$
\item $\left(k,\lfloor\frac{k^2-2}{4}\rfloor\right)$ is $(+1,\frac{12}2-1)$-feasible, for every integer $k\geq 3$.  
\end{itemize}
Moreover, if we write the pairs above as $(k,\widetilde q_s(k,\nu))$ for $(s,\nu)=(-,0),(-,3),(-,11),(+,5)$, respectively, then
$$
k=\A_s^{\textup{disc}}(\widetilde q_s(k,\nu),\nu) + o(k),\text{ as } k\to\infty.
$$
\end{conjecture}

Noting that $j_n\sim \pi n$, as $n\to\infty$, Conjectures \ref{conj:discHTpropertyconj} and \ref{conj:discHTfeaspairs} would imply that $\A_-(8)=\A_+(12)=\sqrt{2}$ and $\A_-(24)=2$, which are known to be true, but also that $\A_-(2)=(\frac 43)^{\frac14}$, which is the content of Conjecture 
\ref{conj:A-1ALP2dconj}.

\subsection{Hamming Cube}\label{sec:Ham}
The Hamming cube $H_N:=\{-1,1\}^N$ 
can be equipped with normalized counting measure, $\lambda_H:=2^{-N}\#$, and the Hamming distance $d_H:H_N\times H_N\to[N]$, 
\[{d}_H(x,y):=\#\{n\in[N]: x_n\neq y_n\}.\]
We write $x=(x_1,\ldots,x_N)\in H_N$ with $x_j=\pm 1$, for each $j$, 
and let ${\bf 1}=(1,\ldots,1)\in H_N$.
An orthonormal basis of $L^2(H_N)=L^2(H_N,\lambda_H)$ is given by $\{\varphi_S: S\subseteq [N]\}$, where $\varphi_S:H_N\to\{-1,1\}$ are the monomials defined via $\varphi_S(x):=\prod_{i\in S} x_i$, with the understanding that $\varphi_\emptyset\equiv 1$. Every function $f:H_N\to\R$ admits an expansion of the form
\[f=\sum_{S\subseteq [N]} \widehat{f}(S) \varphi_S,\]
with (real-valued) coefficients given by 
\[\widehat{f}(S):=\frac1{2^N}\sum_{x\in H_N} f(x) \varphi_S(x).\]
Let $\ft H_N=\{c:2^{[N]}\to \R\}$ denote the finite dimensional vector space of sequences of real numbers indexed by subsets of $[N]$, and define 
$$
\|c\|_{L^2(\ft H_N)}^2:=\frac{1}{2^N}\sum_{S\subseteq [N]}|c(S)|^2.
$$
The operator $T:H_N\to \ft H_N, f \mapsto 2^\frac{N}{2}\ft f,$ defines an isometric isomorphism, in the sense that
$$
\|T (f)\|^2_{L^2(\ft H_N)}=\sum_{S\subseteq [N]} |\ft f(S)|^2 = \|f\|_{L^2(H_N)}^2.
$$
Moreover, $\sup_{S\subseteq [N]} |T (f)(S)| \leq 2^{\frac{N}{2}}\|f\|_{L^1(H_N)}$. 
We can then apply Theorem \ref{thm:OSUP} to the operator $T$, with $p=q=2$, $a=2^{\frac{N}{2}}$, and $b=c=1$, and obtain the following result.

\begin{theorem}\label{thm:hamming}
Let $s\in\{+,-\}$. Let $f:H_N\to \R$ be nonzero, and such that 
$$
\sum_{x\in H_N} f(x)\leq 0, \ \, \ \ sf({\bf 1}) \leq 0.
$$
Then the following estimate holds:
$$
\#\{x\in H_N:f(x)<0\} \cdot \#\{S\subseteq  [N]:s\ft f(S)<0\} \geq  2^{N-4}.
$$
In particular, if $f(x)\geq 0$ when $\,d_H(x,{\bf 1})\geq r$ and $s\widehat{f}(S)\geq 0$ when $\# S\geq k$, then
\begin{equation}\label{eq:HammingGoal}
 \sum_{n=1}^{r} \binom N {n-1} \sum_{n=1}^{k} \binom N{n-1} \geq 2^{N-4}.
\end{equation}
\end{theorem}

\section{Sign Uncertainty for Convolution Operators}\label{sec:ConvolutionOps}

\subsection{Convolution Kernels in Bandlimited Function Spaces}\label{sec:CKBL}
Let $PW_d$ denote the {\it $L^1$-Paley--Wiener space of bandlimited functions in $\R^d$}, i.e.\@ the set of all real-valued, continuous functions $f\in L^1(\R^d)$, whose Fourier support is contained on the unit ball, ${\rm supp}(\ft f)\subseteq B_1^d$.
Given a function $\psi:\R^d\to \R$ for which $\ft \psi(0)\geq 0$ and there exist  $a,b,c\in(0,\infty)$, such that $\|\psi\|_{L^\infty}=a$, $\|\psi\|_{L^1}=b$, and $c|\ft \psi(\xi)|\geq 1$ if $\xi\in B_1^d$,
consider the associated convolution operator,
$T_\psi(f):=f\ast\psi.$
 Young's convolution inequality  and Plancherel's Theorem together imply that 
 $\|T_\psi(f)\|_{L^\infty}\leq a\|f\|_{L^1}$, 
 $\|T_\psi(f)\|_{L^1}\leq b\|f\|_{L^1}$, 
 $\|T_\psi(f)\|_{L^2}\leq b\|f\|_{L^2}$, and 
 $\|f\|_{L^2}\leq c\|T_\psi(f)\|_{L^2}$, for every $f\in PW_d$. Therefore the family $\F=\{(f,T_\psi(f)):f\in PW_d\}$ satisfies the hypotheses of Theorem \ref{thm:OSUP} with $p=q=2$, and we obtain the following result.

\begin{theorem}\label{thm:PW}
Let $d\geq 1$. Let $\psi:\R^d\to\R$ be as above. Let $f\in PW_d\setminus\{{\bf 0}\}$ be such that $\int_{\R^d} f \leq 0$. Then the following inequality holds:
$$
|\{x\in\R^d: f(x)<0\}||\{\xi\in\R^d: T_\psi (f)(\xi) <0\}| \geq (16a^2b^2c^4)^{-1}.
$$
In particular, if there exist $r_1,r_{2}>0$ such that $f(x)\geq 0$ if $|x|\geq r_1$, and $T_\psi(f)(\xi)\geq 0$ if $|\xi|\geq r_{2}$, then
$$
r_1 r_{2} \geq \left({16a^2b^2c^4|B_1^d|^2}\right)^{-\frac1d}.
$$
\end{theorem}
Theorem \ref{thm:PW} can be extended to the more general setting of locally compact abelian groups; the reader is referred to \cite{R62} for the relevant background.

\subsection{Hilbert Transform of Bandlimited Functions}\label{sec:convolution} 
It is of interest to consider the situation in which the kernel $\psi$ from \S \ref{sec:CKBL} above fails to be integrable. For instance, if $d=1$, then the choice $\psi(x)=\frac{1}{\pi x}$ leads to  the Hilbert transform  $\H$, as long as the  convolution is taken in the principal value sense. It is well-known that $\H$ defines a bounded operator in $L^p(\R)$, for all $p\in(1,\infty)$, and that the optimal constant in $\|\H (f)\|_{L^p}\leq C_p \|f\|_{L^p}$ is given by
\begin{equation}\label{eq:Cp}
C_p := \left\{ \begin{array}{ll}
\tan(\frac{\pi}{2p}),& \textrm{if $1<p\leq 2$},\\
\cot(\frac{\pi}{2p}),& \textrm{if $2<p<\infty$};
\end{array} \right.
\end{equation}
see \cite{Pi72}. Moreover, since $\ft{\H (f)}(\xi)=-i\,{\rm sign}(\xi)\ft f(\xi)$, we have that $\H(\H (f))=-f$, hence the reverse inequality, $\|f\|_{L^p}\leq C_p\|\H (f)\|_{L^p}$, holds with the same optimal constant. Now, if $f\in PW_1$ (recall the definition in \S \ref{sec:CKBL}), then $\ft f$ is supported in $[-1,1]$, and consequently
$$
\|\H(f)\|_{L^\infty} \leq \|\ft{\H(f)}\|_{L^1} = \|\ft f\|_{L^1} \leq 2 \|\ft f\|_{L^\infty} \leq 2 \|f\|_{L^1}.
$$
Note that $\ft f$ is continuous since $f\in L^1$. 
A necessary condition for $\H (f)$ to be integrable if $f\in L^1$ is that $\ft f(0)=0$, in which case
$\ft{\H (f)}(0)=0$ as well. 
We then conclude that  
$$
\F_s=\{(f,s\,\H(f))\, : \,  f\in PW_1 \, ; \, \ft f(0)=0\}
$$
satisfies all the hypotheses of Theorem \ref{thm:OSUP}, with $p=q\in(1,\infty)$, $a=2$, and $b=c=C_p$. 
As a consequence, we obtain the following result.

\begin{theorem}\label{thm:Hilbert}
Let $s\in\{+,-\}$ and $p\in(1,\infty)$. Let $f\in PW_1$ satisfy $\ft f(0)=0$. Suppose that there exist $r_1,r_{2,s}>0$, such that $f(y)\geq 0$ if $|y|\geq r_1$, and $s\H(f)(x)\geq 0$ if $|x|\geq r_{2,s}$. 
Then the following estimate holds:
$$
r_1^{1/p'}r_{2,s}^{1/p} \geq {2^{-(p'+2)} C_p^{-\frac{p+1}{p-1}}},
$$
where $C_p$ is given by \eqref{eq:Cp} above.
\end{theorem}
Theorem \ref{thm:Hilbert} can probably be extended to a certain class of singular integral operators given by Calder\'on--Zygmund kernels of convolution type (see \cite[Ch.\@ 5]{Gr08}) which includes the higher dimensional Riesz transforms.

\subsection{Hankel Transform}\label{sec:hankeltransform}
The Hankel transform with parameter $\nu>-1$ of a function $f:\R_+\to \R$ is given by
\begin{equation}\label{def:hankeltrans}
H_\nu (f)(x)=\int_0^\infty f(y) A_\nu (x y) y^{2\nu+1}\d y,
\end{equation}
where $A_\nu(z)=\Gamma(\nu+1) (\frac12 z)^{-\nu} J_\nu(z)$, and $J_\nu$ is the Bessel function of the first kind. Alternative ways to define the Hankel Transform exist, the most common one having $A_\nu$ replaced by $J_\nu$, and $y^{2\nu+1}\d y$ replaced by $y\d y$; recall \eqref{eq:HankelPreDisc}, and see e.g.\@ \cite{T37}.  However, the choice of kernel in \eqref{def:hankeltrans} suits us better since the function $A_\nu(z)$ is entire, $A_\nu(0)=1$, and routine computations show that, if $f:\R^d\to\R$ is radial, then its Fourier transform $\ft f$, as defined in \eqref{def:fouriertrans}, is also radial, and satisfies
\begin{equation}\label{hankelfourier}
\ft f(\xi) = c_d H_{\frac{d}{2}-1}(f)(2\pi |\xi|),
\end{equation}
for some $c_d>0$.
The analogue of \eqref{besselidentity} over the unbounded region of integration $(0,\infty)$ reveals the following Plancherel-type identity:
$$
\int_0^\infty |H_\nu(f)(x)|^2 x^{2\nu+1}\d x
=4^\nu\Gamma^2(\nu+1) \int_0^\infty |f(y)|^2 y^{2\nu+1}\d y.
$$
Moreover, since  $|A_\nu(x)|\leq A_\nu(0)=1$, we easily obtain that
$$
\sup_{x>0}|H_\nu(f)(x)|\leq \int_0^\infty |f(y)|y^{2\nu+1}\d y.
$$
Therefore, for a given $s\in\{+,-\}$, the family 
\begin{align*}
\F_s = \bigg\{(f,H_\nu(f)) \ :\ & f,H_\nu(f)\in L^1(\R_+,y^{2\nu+1}\d y), \\ & \int_{0}^\infty f(y)y^{2\nu+1}\d y, s\int_{0}^\infty H_\nu (f)(x)x^{2\nu+1}\d x\leq 0 \bigg\}
\end{align*}
satisfies the hypotheses of Theorem \ref{thm:OSUP} when $p=q=2$, $a=1$, and $b=1/c=2^\nu\Gamma(\nu+1)$. 
It is then straightforward to derive the following result.

\begin{theorem}\label{thm:Hankel}
Let $s\in\{+,-\}$ and $\nu>-1$. Let $f:\R_+\to\R$ be a continuous nonzero function, such that $f,H_\nu(f)\in L^1(\R_+,y^{2\nu+1}\d y)$. Assume that there exist $r_1,r_{2,s}>0$, such that $f(y)\geq 0$ if $y\geq r_1$ while $H_\nu(f)(0)\leq 0$, and $s H_\nu(f)(x)\geq 0$ if $x\geq r_{2,s}$ while $s f(0)\leq 0$. 
Then the following estimate holds:
$$
r_1 r_{2,s} \geq 4^{\nu-2}\Gamma^2(\nu+1).
$$
\end{theorem}

\section{Proofs of Main Results}\label{sec:PfThms}

\subsection{Proof of Theorem \ref{thm:OSUP}}
\begin{proof}
Since $\int_X f\d\mu\leq 0$, we have that 
\begin{equation}\label{eq:L1est}
\|f\|_{L^1(X,\mu)}\leq 2\int_{\{f<0\}} |f| \d\mu\leq 2 \mu(\{f<0\})^{\frac1{p'}} \|f\|_{L^p(X,\mu)},
\end{equation}
where the last estimate follows from H\"older's inequality.
On the other hand, the hypotheses, convexity of $L^p$-norms, the fact that $s\int_Y g\d\nu\leq 0$, and a second application of H\"older's inequality, together yield
\begin{align*}
\|f\|_{L^p(X,\mu)}^q
& \leq c^q\|g\|_{L^q(Y,\nu)}^q \\
& \leq c^q\|g\|_{L^\infty(Y,\nu)}^{q-1} \|g\|_{L^1(Y,\nu)}\\
&\leq 2c^qa^{q-1} \|f\|_{L^1(X,\mu)}^{q-1} \int_{\{s g<0\}} |g|\d\nu \\
&\leq 2c^qa^{q-1} \|f\|_{L^1(X,\mu)}^{q-1} \nu(\{s g<0\})^{\frac1{q'}} \|g\|_{L^q(Y,\nu)} \\
&\leq 2c^qa^{q-1}b \|f\|_{L^1(X,\mu)}^{q-1} \nu(\{s g<0\})^{\frac1{q'}}  \|f\|_{L^p(X,\mu)}. 
\end{align*}
Cancelling one power of $\|f\|_{L^p(X,\mu)}$ (which is allowed since $f$ is nonzero), taking the $(q-1)$-th root on both sides, and plugging the resulting estimate into \eqref{eq:L1est}, we finally obtain:
\[\|f\|_{L^1(X,\mu)}\leq ab^{\frac{q'}q}(2c)^{q'}\mu(\{f<0\})^{\frac1{p'}}  \nu(\{sg<0\})^{\frac1q}\|f\|_{L^1(X,\mu)},\]
from where \eqref{eq:OSUP} follows at once.
\end{proof}

\subsection{Proof of Theorem \ref{thm:MSUP}}
\begin{proof}
Let $f\in\mathcal A_s(X)\setminus\{{\bf 0}\}$ and $S:=\{x\in X: f(x)<0\}$.
On the one hand,
\[0\geq\widehat f(0)=\int_X f \d\lambda=\int_{X\setminus S} |f|\d\lambda-\int_{S}|f|\d\lambda,\]
and therefore
\begin{equation}\label{eq:ToBeUsed}
\|f\|_{L^1(X)}\leq 2\int_{S} |f|\d\lambda\leq 2\lambda(S)^{\tfrac12}\|f\|_{L^2(X)}.
\end{equation}
On the other hand, setting $R:=\{n\geq 0: s\ft f(n)<0\}$, we have
\begin{equation}\label{eq:Second}
0\geq sf(\mathfrak 0)=\sum_{n=0}^\infty s\widehat f(n)\varphi_n(\mathfrak 0)
=\sum_{n \notin R} |\widehat f(n)| \|\varphi_n\|_{L^\infty(X)}
-\sum_{n \in R} |\widehat f(n)| \|\varphi_n\|_{L^\infty(X)},
\end{equation}
where in the latter identity we used that $\varphi_n(\mathfrak 0)=\|\varphi_n\|_{L^\infty(X)}$. We also have that
\[|\widehat f (n)|=\left| \int_X f\varphi_n\d\lambda\right|\leq\|f\|_{L^1(X)}\|\varphi_n\|_{L^\infty(X)},\]
and therefore
\begin{align*}
\|f\|_{L^2(X)}^2
&=\sum_{n=0}^\infty |\widehat f(n)|^2
\\ & \leq  \|f\|_{L^1(X)}\sum_{n=0}^\infty |\widehat f(n)| \|\varphi_n\|_{L^\infty(X)}\\
&\leq  2\|f\|_{L^1(X)}\sum_{n\in R} |\widehat f(n)| \|\varphi_n\|_{L^\infty(X)}
\\ & \leq  2\|f\|_{L^1(X)}\|f\|_{L^2(X)}\left(\sum_{n\in R} \|\varphi_n\|_{L^\infty(X)}^2\right)^{\tfrac12}.
\end{align*}
From the second to the third lines, we appealed to \eqref{eq:Second}.
Cancelling one power of $\|f\|_{L^2(X)}$ from both sides, and plugging the resulting estimate into \eqref{eq:ToBeUsed}, yields \eqref{eq:MSUP1}.
\end{proof}

\subsection{Proof of Theorem \ref{thm:Sphere}}\label{sec:PfThmSph}
\begin{proof}
The strategy is to establish identity \eqref{eq:spheregegenid}, and then invoke Theorem \ref{thm:Jacobi}.
With this purpose in mind, let $f\in\mathcal B_s(\sph{d-1})\setminus\{{\bf 0}\}$, and 
let $\textup{SO}_{\eta}(d)\subseteq \textup{SO}(d)$ denote the subgroup of rotations which fix the north pole $\eta\in\mathbb{S}^{d-1}$, equipped with Haar probability measure  $\gamma$.
Consider the {\it partially radialized} function $g:\sph{d-1}\to\R$, defined by
\begin{equation}\label{eq:RadialF}
g(\omega)=\int_{\textup{SO}_{\eta}(d)} f(\rho\,\omega) \d\gamma(\rho).
\end{equation}
One easily checks that $g$ is continuous,  $sg(\eta)=sf(\eta)\leq 0$,
and that $\theta(g)\leq\theta(f)$. 
Note that the possibility that $g\equiv 0$ cannot be excluded, so we split the analysis into two cases.

First we consider the case when $g$ is nonzero. Set $\nu=\tfrac{d}2-1$, and let $Z_n(\omega):=C_n^{\nu}(\langle \omega,\eta\rangle)$ denote the zonal harmonic of degree $n$.
Here, $C_n^{\nu}$ is the Gegenbauer polynomial of degree $n$; see \eqref{eq:JacobivsGegen}.
If $d\geq 3$, then $\frac{n+\nu}{\nu} C_n^{\nu} (\langle\cdot,\cdot\rangle)$ is the reproducing kernel of $\mathcal H_n^{d}$ with respect to the normalized surface measure on $\sph{d-1}$; see \cite[Def.\@ 1.2.2 and Theorem 1.2.6]{DX13}. 
Consequently,
\begin{equation}\label{eq:ReprKern}
\int_{\textup{SO}_{\eta}(d)} P(\rho\,\omega)\d\gamma(\rho)=P(\eta)\frac{Z_n(\omega)}{Z_n(\eta)}, \text{ for every } P\in\mathcal H_n^{d}.
\end{equation}
To verify identity \eqref{eq:ReprKern}, one checks that the left-hand side depends on $\omega$ only through its inner product with the north pole, invokes \cite[Lemma 1.7.1]{DX13}, and  sets $\omega=\eta$ to compute the leading constant on the right-hand side.
It follows from \eqref{eq:ExpandFSph}, \eqref{eq:RadialF}, \eqref{eq:ReprKern} that
\[g(\omega)=\sum_{n=0}^\infty a_n Z_n(\omega),\text{ where }a_n:=\sum_{j=1}^{h_n} \widehat f(n,j) \frac{Y_{n,j}(\eta)}{Z_n(\eta)}.\]
From \eqref{eq:JacobiNorm} and \eqref{eq:JacobivsGegen}, we have that
$Z_n(\eta)=C_n^{\nu}(1)
=\binom{n+2\nu-1}n>0$,
and since the basis $\{Y_{n,j}\}$ is signed, it follows that $sa_n\geq 0$, for every $n\geq k(s\widehat f)$.
Set $G(x):=g(\omega)$, where $x=\langle \omega,\eta\rangle$.
The function $G:[-1,1]\to\R$ is continuous, and satisfies $sG(1)=sg(\eta)\leq 0$.
Moreover, for every $x\in[-1,\cos(\theta(f))]$, we have that $G(x)=\sum_{n=0}^\infty a_n C_n^{\nu}(x)\geq 0$, 
where  $sa_n\geq 0$, for every $n\geq k(s\ft f)$.
As a consequence, we obtain the following lower bound:
\begin{align}\label{eq:BsjacBsphineq}
(1-\cos(\theta(f)))k(s\widehat f)^2
 \geq \mathbb{B}_s(I;\nu-\tfrac12,\nu-\tfrac12).
\end{align}

If $g\equiv 0$, then $a_n=0$ for all $n\geq 0$, and since $Y_{n,j}(\eta)>0$ for all sufficiently large $n$, we also have that $\ft f(n,j)=0$ for all sufficiently large  $n$.
Hence $f$ is a polynomial. 
In turn, this implies $\theta(f)=\pi$, for otherwise $f$ would have to vanish identically on the spherical cap $\{\om\in\sph{d-1}: \theta(f)<d_g(\om,\eta)\leq\pi\}$, which cannot happen unless $f$ were the zero polynomial. 
This shows that $(1-\cos(\theta(f)))k(\widehat f)^2 \geq 2$  and\footnote{Recall that, by the discussion preceding the statement of Theorem \ref{thm:MSUP}, we must have $k(-\widehat f)\geq 2$.} $(1-\cos(\theta(f)))k(-\widehat f)^2 \geq 8$.  
On the other hand, the functions 
$$
f_+(\om)=-1+\frac{C_1^\nu(x)}{C_1^\nu(1)}, \,\,\, 
f_-(\om)=-\frac{C_1^\nu(x)}{C_1^\nu(1)}+\frac{C_2^\nu(x)}{C_2^\nu(1)},
$$
respectively belong to $\b_+(\sph{d-1})$, $\b_-(\sph{d-1})$  as functions of $\om$, and respectively belong to $\b_+(I;\nu-\tfrac12,\nu-\tfrac12) $, $\b_-(I;\nu-\tfrac12,\nu-\tfrac12) $  as functions of $x=\langle \omega,\eta\rangle$.
 They also satisfy $(1-\cos(\theta(f_+)))k(\widehat f_+)^2 = 2$ and $(1-\cos(\theta(f_-)))k(-\widehat f_-)^2 =8$, hence \eqref{eq:BsjacBsphineq} still holds.
 This also establishes the upper bounds in \eqref{eq:upperbdsBsph}.
  We conclude that $\mathbb{B}_s(\sph{d-1})^2\geq \mathbb{B}_s(I;\nu-\tfrac12,\nu-\tfrac12)$.
Conversely, given a function $F$ in $\mathcal{B}_s(I;\nu-\tfrac12,\nu-\tfrac12)$, then $f:= F(\langle \cdot,\eta\rangle)$ belongs to $\mathcal{B}_s(\sph{d-1})$, and satisfies
$$
(1-\cos(\theta(f)))^\frac12k(s\widehat f) = r(F;I)^\frac12 k(s\ft F).
$$
This shows that $\mathbb{B}_s(\sph{d-1})^2\leq \mathbb{B}_s(I;\nu-\tfrac12,\nu-\tfrac12)$, and therefore \eqref{eq:spheregegenid} holds. 

Theorem \ref{thm:Jacobi} then implies the following lower bound:
\begin{align*}
\mathbb B_s(\mathbb{S}^{d-1})
& = \mathbb{B}_s(I;\nu-\tfrac12,\nu-\tfrac12)^{\frac12}  \\ 
& \geq \left[\frac{\Gamma(\nu+\frac32)^{\frac2{\nu+1/2}}}{(4e^{\frac1{12}})^{\frac2{\nu+1/2}}(\nu+\tfrac12)(\nu+\tfrac32)}
\right]^\frac12  
=\frac{2\Gamma(\frac{d+1}2)^{\frac2{d-1}}}{(4e^{\frac1{12}})^{\frac2{d-1}}{(d^2-1)^{\frac12}}}.
\end{align*}
This concludes the proof of the theorem.
\end{proof}

\subsection{Proof of Theorem \ref{thm:Jacobi}}
\begin{proof}
Let $\alpha\geq \beta\geq-\tfrac12$.
Consider the interval $I=[-1,1]$, equipped with the restricted Euclidean metric $d$ and the probability measure $w_{\alpha,\beta}$. Then $(I,d,w_{\alpha,\beta})$ is an admissible space in the sense of Definition \ref{def:AdmSp}, with $\frak 0=1$.
Indeed, if $\alpha=\max\{\alpha,\beta\}\geq-\tfrac12$,
then from \cite[Theorem 7.32.1]{Sz75} and \eqref{eq:JacobiNorm} it follows that  
\begin{equation}\label{eq:maxJac}
\max_{-1\leq x\leq 1} |P_n^{(\alpha,\beta)}(x)|=P_n^{(\alpha,\beta)}(1),
\end{equation}
and therefore the orthogonal basis $\{p_n^{(\al,\beta)}\}_{n\in\N}$ of $L^2(I)$ satisfies \eqref{eq:OriginVsInfty} with $\frak 0=1$.

Moreover, the class $\mathcal A_s(I)$ from Definition \ref{def:AsCone} coincides with the class $\mathcal B_s(I;\alpha,\beta)$ from Definition \ref{def:AsJac}.
To see why this is the case, note that \eqref{eq:maxJac} and the second condition required by Definition \ref{def:AsCone} together imply that
\begin{equation}\label{eq:FiniteSumJac}
\sum_{n=0}^\infty |\widehat f(n)| {p_n^{(\alpha,\beta)}(1)}<\infty.
\end{equation}
Therefore the series \eqref{eq:JacobiSeries} converges absolutely and uniformly, and the function $f$ is continuous.
This shows that $\mathcal A_s(I)\subseteq \mathcal B_s(I;\alpha,\beta)$.
Conversely, the sequence $\{s\widehat f(n)\}_{n\in\N}$ being eventually nonnegative implies that \eqref{eq:FiniteSumJac} holds if and only if
$
\sum_{n=0}^\infty \widehat f(n) {p_n^{(\alpha,\beta)}(1)}<\infty,
$
which in turn is equivalent to the limit
$\lim_{r\to 1^-}\sum_{n=0}^\infty \widehat f(n) {p_n^{(\alpha,\beta)}(1)} r^n$
existing and being finite.
The latter limit exists and equals $f(1)$ since the power series of any real-valued, continuous function on $I$ is Abel summable.
It follows that $\mathcal A_s(I)=\mathcal B_s(I;\alpha,\beta)$, as claimed.

From Theorem \ref{thm:MSUP}, it then follows directly that
\begin{equation}\label{eq:MSUPimpliesJac}
\left( \int_{1-r(f;I)}^1 w_{\al,\beta}(x)\d x \right)
\sum_{n=1}^{k(s\widehat f)} \frac{P_{n-1}^{(\alpha,\beta)}(1)^2}{h_{n-1}^{(\alpha,\beta)}} 
\geq \frac1{16}.
\end{equation}
To estimate the left-hand side of \eqref{eq:MSUPimpliesJac}, start by noting that the confluent form of the Christoffel--Darboux formula for Jacobi polynomials (see \cite[(4.5.8)]{Sz75}) implies that
\begin{align}
\sum_{n=1}^{k(s\widehat f)} \frac{P_{n-1}^{(\alpha,\beta)}(1)^2}{h_{n-1}^{(\alpha,\beta)}} 
=\frac{\Gamma(\alpha+k(s\widehat f)+1)\Gamma(\alpha+\beta+k(s\widehat f)+1)\Gamma(\beta+1)}{\Gamma(\alpha+2)\Gamma(k(s\widehat f))\Gamma(\beta+k(s\widehat f ))\Gamma(\alpha+\beta+2)}.\label{eq:CDarboux}
\end{align}
A version  of Stirling's formula 
for the Gamma function \cite{Ro55} states that
\begin{equation}\label{eq:Stirling}
\Gamma(x)=\sqrt{2 \pi} x^{x-\tfrac12} e^{-x}e^{\mu(x)}, \text{ for every } x> 0,
\end{equation}
where the function $\mu$ satisfies the two-sided inequality
$\frac 1{12x+1}< \mu(x)< \frac 1{12x}.$
Moreover, it is elementary to check that
\[\left(1+\frac ax\right)^x\leq \exp(a),\text{ for every }a,x\geq 0.\]
In particular, if $x\geq y\geq -1, k\geq 1$, then we may estimate:
\begin{align*}
\frac{\Gamma(k+x+1)}{\Gamma(k+y+1)}
&\leq e^{\frac1{12}}\frac{(k+x+1)^{k+x+\tfrac12}e^{-k-x-1}}{(k+y+1)^{k+y+\tfrac12}e^{-k-y-1}}\\
&=e^{\frac1{12}} e^{y-x} (k+x+1)^{x-y} \left(1+\frac{x-y}{k+y+1}\right)^{k+y+\tfrac12}\\
&\leq e^{\frac1{12}}(k+x+1)^{x-y}\leq e^{\frac1{12}} k^{x-y}(x+2)^{x-y}.
\end{align*}
Applying the latter estimate (twice) to \eqref{eq:CDarboux}, with $k=k(s\widehat f)$, yields
\begin{multline}\label{eq:GammaEst}
\frac{\Gamma(\alpha+k(s\widehat f)+1)\Gamma(\alpha+\beta+k(s\widehat f)+1)\Gamma(\beta+1)}{\Gamma(\alpha+2)\Gamma(k(s\widehat f))\Gamma(\beta+k(s\widehat f))\Gamma(\alpha+\beta+2)}\\
\leq
\frac{e^{\frac16}(\alpha+2)^{\alpha+1}(\alpha+\beta+2)^{\alpha+1}\Gamma(\beta+1)}{\Gamma(\alpha+2)\Gamma(\alpha+\beta+2)}k(s\widehat f)^{2\alpha+2}.
\end{multline}
On the other hand, a crude estimate together with identity \eqref{eq:Defcab} yield 
\begin{equation}\label{eq:CrudeEst}
 \int_{1-r(f;I)}^1 w_{\al,\beta}(x)\d x 
\leq c_{\alpha,\beta} 2^\beta \int_{1-r(f;I)}^1 (1-x)^\alpha\d x
=\frac1{2^{\alpha+1}}\frac{\Gamma(\alpha+\beta+2)}{\Gamma(\alpha+2)\Gamma(\beta+1)} r(f;I)^{\alpha+1}.
\end{equation}
The lower bound in \eqref{eq:2sidedJac} now follows from \eqref{eq:MSUPimpliesJac}, \eqref{eq:CDarboux}, \eqref{eq:GammaEst},  \eqref{eq:CrudeEst}.
Since the upper bounds were already established via \eqref{eq:f+f-}, this concludes the proof of the theorem.
\end{proof}

\subsection{Proof of Proposition \ref{prop:WLOG}}
\begin{proof}
We split the proof into the cases $s\in\{+,-\}$.

{\bf Case} $s=-1$.
Let $f\in \b_-(I;\al,\be)\setminus \{{\bf 0}\}$, and consider the auxiliary polynomial $g_-$, 
$$
g_-(x)=\frac{(1-x_{1,n})}{p_n^{(\al,\beta)}(1)^2}\frac{p_n^{(\al,\beta)}(x)^2}{(x-x_{1,n})},
$$
where $x_{1,n}$ denotes the largest zero\footnote{More generally, we let $-1<x_{n,n}<x_{n-1,n}<\ldots<x_{1,n}<1$ denote the zeros of the polynomial $p_n^{(\al,\beta)}$.} of $p_n^{(\al,\beta)}$. 
Clearly, $g_-(1)=1$, $g_-(x)\leq 0$ if $-1\leq x\leq x_{1,n}$, and $\ft {g}_-(0)=0$ (since $p_n^{(\al,\beta)}$ is orthogonal to all polynomials of degree less than $n$). We claim that $\ft {g}_-(n)\geq 0$, for all $n\geq 1$. 
Indeed, \cite[Theorem]{G70} states that, for all $m,n\geq 0$,
$$
p_n^{(\al,\beta)}(x)p_m^{(\al,\beta)}(x) = \sum_{j=0}^{m+n} R(\al,\be,j)p_j^{(\al,\beta)}(x),
$$
where $R(\al,\be,j)\geq 0$, for $j=0,\ldots,m+n$. 
 Moreover, \cite[Theorem 3.1]{CK07} implies that the Jacobi expansion of the polynomial
\[x\mapsto\frac{p_n^{(\al,\beta)}(x)}{\prod_{j=1}^\ell(x-x_{j,n})},\,\,\, (1\leq\ell\leq n)\] 
has nonnegative coefficients. 
Together these results directly imply the claim.  
Since, for any fixed $\ell$,  $x_{\ell,n}\to 1$ as $n\to\infty$, one can set
$F_-:= f -f(1)g_-,$
and check that $F_-\in \b^0_-(I;\al,\be)\setminus \{{\bf 0}\}$, $k(-\ft F_-)=k(-\ft f)$, $r(F_-;I)< r(f;I)$, provided $n$ is chosen sufficiently large.  

{\bf Case} $s=+1$. Let $f\in \b_+(I;\al,\be)\setminus \{{\bf 0}\}$,  and consider the auxiliary polynomial $g_+$,
$$
g_+(x)=\frac{(1-x_{1,n})(1-x_{2,n})}{p_n^{(\al,\beta)}(1)^2}\frac{p_n^{(\al,\beta)}(x)^2}{(x-x_{1,n})(x-x_{2,n})}.
$$
Similarly to the case $s=-1$, we have that  $g_+(1)=1$, $g_+(x)\geq 0$ if $-1\leq x\leq x_{2,n}$, $\ft {g}_+(0)=0$, and $\ft {g}_+(n)\geq 0$ for all $n\geq 1$. Letting $F_+:=f -f(1)g_+$, we check that $F_+\in\b_+^0(I;\al,\beta)\setminus\{{\bf 0}\}$, satisfies $k(\ft F_+)=k(\ft f)$, $r(F_+;I)<r(f;I)$, provided $n$ is chosen sufficiently large.
\end{proof}

\subsection{Proof of Theorem \ref{thm:localoptuadrature}}
We present the proof for the polynomial $P$ only, since it proceeds analogously for $Q$. 
For simplicity, we write $x_0=1$ and $\{x_m<...<x_1\}\subset (-1,1)$ for the zeros of the polynomial $p_{m}^{(\al+1,\be)}$.
The crux of the matter boils down to the following simple result. 

\begin{lemma}\label{lem:lowerboundquad} 
Let $f\in\b_s(I;\al,\beta)\setminus \{{\bf 0}\}$ be a polynomial of degree at most $2m$, and further assume that $f(1)=0$ if $s=+1$. Then $r(f;I)\geq 1-x_1$, where equality is attained if and only if $f$ is a positive multiple of the polynomial $P$ in \eqref{eq:localoptpolyquadrature}. 
\end{lemma}

\begin{proof}[Proof of Lemma \ref{lem:lowerboundquad}] 
Aiming at a contradiction, assume that $r(f;I)<1-x_1$. 
Then $f(x)\geq 0$ if $-1\leq x\leq x_1$, whence
$$
0\leq \la_0 f(1)+ \sum_{j=1}^m \la_j f(x_j) =  \int_{-1}^1 f(x) w_{\al,\beta}(x)\d x=\ft f(0)\leq 0.
$$
Thus $f(x_j)=0$ for $j=0,\ldots,m$, and $f'(x_j)=0$ for $j=1,\ldots,m$. Moreover, $f$ necessarily vanishes at $x=1-r(f;I)$. We conclude that $\deg(f)\geq 2m+2$, which is absurd. 
The preceding argument further shows that if $r(f;I)=1-x_1$, then $f$ must coincide with a positive multiple of the polynomial \eqref{eq:localoptpolyquadrature}. 
\end{proof}

\begin{proof}[Proof of Theorem \ref{thm:localoptuadrature}]
Set $k:=k(s\ft P)$. 
Note that $k\geq 2$, and that $s\ft P(k-1)<0$. 
Moreover, since $P$ is monic of degree $2m$, then $k=2m+1$ if $s=-1$. 
Set $\delta:=-\frac12 s\ft P(k-1)$, and let $h\in\b_s^0(I;\al,\beta)\setminus\{{\bf 0}\}$ be such that $\|ch-P\|_{L^\infty(I)}<\delta$, for some $c>0$. Estimate:
$$
|c\ft h(k-1)-\ft P(k-1)| \leq \|c h- P\|_{L^2(I)} \leq \|ch-P\|_{L^\infty(I)}<\delta = -\frac12 s\ft P(k-1).
$$
Thus $sc\ft h(k-1)<\frac12 s\ft P(k-1) <0$, and $k(s\ft h)\geq k$. Lemma \ref{lem:lowerboundquad} implies that if  $h$ is not a multiple of $P$ (i.e.\@ $\inf_{c>0} \|ch-P\|_{L^\infty(I)}>0$), then $r(P;I)<r(h;I)$. Therefore
$r(P;I)k(s\ft P)^2 < r(h;I)k(s\ft h)^2,$
as desired.
\end{proof}

\section{Numerical Evidence}\label{sec:numerics}

\subsection{Discrete Fourier Transform}\label{numerics:DFT}
Conjecture \ref{conj:discFTpropertyconj} implies the existence of a well-defined jump function $k\mapsto q_s(k)$, which records the smallest value of $q$ for which $(k,q)$ is $s$-feasible but $(k-1,q)$ is not; in other words, $k=\A_s^{\textup{disc}}(q_s(k))$, and no other $q<q_s(k)$ has this property. We strongly believe that the first few values of $q_s(k)$ coincide with the ones displayed in Table \ref{table:1}, although we cannot claim its correctness in any rigorous way since all the computations were performed using floating-point arithmetic. In the case $s=-1$, the pattern of $q_s(k)$ in  Table \ref{table:1} is easy to guess, since for $k>3$ it is in perfect accordance with the  sequence
$$
\left\lceil \frac{(k-1)^2}{2}\right\rceil_{k\geq 4} =  5, 8, 13, 18, 25, 32, 41, 50, 61, 72,\ldots.
$$ 
From Proposition \ref{prop:discguess-1} we know that $\A_-^{\rm disc}(q)\leq \ell$ if  $q=(\ell^2-1)/2$. However, $(\ell^2-1)/2$ is never equal to $\left\lceil \frac{(k-1)^2}{2}\right\rceil$, and this is why we see no entry equal to $1$ in the column of Table \ref{table:1} corresponding to $\tfrac{k}{\sqrt{2q_-+1}}$. 

In the case $s=+1$, the pattern is not so easy to guess, although it  seems to grow quadratically with $k$. Surprisingly, typing the numbers $6,14,25,40,58$ into the {\it On-Line Encyclopedia of Integer Sequences} \cite{oeis} returns precisely one hit, which reveals that our numerical approximation of $q_+(k)$ agrees for $k\in\{3,4,5,6,7\}$ with
\begin{equation}\label{eq:goldenfloor}
\left\lfloor {(k-1)^2\varphi}\right\rfloor_{k\geq 3}=  6, 14, 25, 40, 58, 79, 103, 131, 161, 195,\ldots,
\end{equation}
where $\varphi=\frac{1+\sqrt{5}}{2}$ denotes the golden ratio. Unfortunately, this coincidence stops at $k=7$, and from then onwards our numerical value of $q_+(k)$ seems to be slightly larger than that of  \eqref{eq:goldenfloor}. 
One might still conjecture that $q_+(k)=\lfloor (k-1)^2\varphi\rfloor+o(k)$ which would show, under Conjecture \ref{conj:discFTpropertyconj}, that $\A_+(1)=(2\varphi)^{-\frac12}=0.5558\ldots$ A least squares fit for the data shows that actually $q_+(k)\approx 0.882-3.348k+1.65 k^2$, which under Conjecture \ref{conj:discFTpropertyconj} suggests that
\[\A_+(1)\approx 0.550.\]
However we can derive a more reliable upper bound for $\A_+(1)$ by exploiting monotonicity. 
Noting that $(k,q)$ is $+$-feasible for any $q$ in the interval $q_+(k)\leq q<q_+(k+1)$, we can look at the function $v(k)=\frac k{\sqrt{2q_+^\ast(k)+1}}$, where $q_+^\ast(k)=\lfloor \frac{q_+(k)+q_+(k+1)-1}2\rfloor$. This function is decreasing for $3\leq k\leq 67$; see  Figure \ref{figure:3}.
If $v(k)$ is decreasing for all $k\geq 3$, then from this and Conjecture \ref{conj:discFTpropertyconj} it would follow that
\[\A_+(1)<v(68)=0.5548\ldots<0.555,\]
as predicted by Conjecture \ref{conj:A+11dconj}. In particular, this rules out the aforementioned relation between $\A_+(1)$ and the golden ratio.

 \begin{figure}[H]
\begin{tikzpicture}
\begin{axis}[
scale=1.6,
axis lines=middle,
xtick={3,35,67},
xticklabel style={font=\tiny},
ytick={0.57735,0.5507},
ymin=0.54,
ymax=0.6,
xmax=67,
xmin=3,
]
\addplot[
  solid,
]
table{v.data};
\end{axis}
\end{tikzpicture}
\caption{This is a plot of the function $v(k)=\frac{k}{\sqrt{2q_+(k+1)-1}}$.}
\label{figure:3}
\end{figure}

The most outstanding feature of our numerics is the possibility that a minimizer for $\A_+(1)$ vanishes identically in certain intervals; see Figure \ref{figure:1}.
The first author together with Henry Cohn and David de Laat have unpublished numerical data in strong support of an upper bound for $\A_+(1)$  which starts with $0.558\ldots$ The function attaining the latter bound is a polynomial multiple of a Gaussian, and exhibits a shape which is remarkably akin to the plot in Figure \ref{figure:1}; in particular, it appears to vanish identically in similar intervals. It is worth pointing out that, since $q_s(k)$ seems to grow quadratically with $k$, the error of ${k}(2q_s(k)+1)^{-\tfrac12}$ is of the order $O(k^{-1})$. Therefore, in order to obtain a ~3-digit approximation of the limit of ${k}(2q_s(k)+1)^{-\tfrac12}$, as $k\to\infty$, one would have to set $k\approx 10^3$ and run several linear programs with $q\approx 10^6$, which lies at the computational limit of what the current best linear programming solvers can accomplish in a reasonable time frame. For some reason which is unclear to us, the $+1$ uncertainty principle consistently seems to be computationally harder than the $-1$ uncertainty principle.

 \begin{figure}[H]
\centering
\begin{tikzpicture}[scale=1.6]
\begin{axis}[
axis lines=middle,
xtick={0,1,2,3,4,5,6,7,8,9,10,11,12},
xticklabel style={font=\tiny},
ytick={0},
ymin=-1.2,
ymax=0.5,
xmax=12,
xmin=0,
]
\addplot[
	solid,
]
table{monster.data};
\end{axis}
\end{tikzpicture}
\caption{This is a plot of the sequence $\left\{\left(\tfrac{n}{\sqrt{2q+1}},f(n)\right)\right\}_{n=0}^q$, where $f$ is an optimal answer to Problem \ref{DiscFeasLP} in the case $s=+1$ with $k_{f}=68$ and $q=7401$. Moreover, this vector satisfies $\ft f=f$, $f(0)=0$, and has minimal energy $\sum_{n=68}^{7401} f(n)^2$. One can only wonder whether the flatter areas in the plot indicate that  minimizers for $\A_+(1)$ may vanish identically in certain intervals.}
\label{figure:1}
\end{figure} 

\begin{table}[H]
\begin{changemargin}{0cm}{-1cm}{-1cm}
\centering
\begin{tabular}{|c|c c|c c||c|c c|c c||c|c c|c c|}
\hline
$k$ & $q_-$ & $\frac{k}{\sqrt{2q_-+1}}$ & $q_+$ & $\frac{k}{\sqrt{2q_++1}}$ & $k$ & $q_-$ & $\frac{k}{\sqrt{2q_-+1}}$ & $q_+$ & $\frac{k}{\sqrt{2q_++1}}$ & $k$ & $q_-$ & $\frac{k}{\sqrt{2q_-+1}}$ & $q_+$ & $\frac{k}{\sqrt{2q_++1}}$\\ [.7ex]
\hline
3 & 3 & 1.3339 & 6 & 0.8321 & 25 & 288 & 1.0408 & 948 & 0.5740 & 47 & 1058 & 1.0215 & 3488 & 0.5627 \\ 
4 & 5 & 1.2060 & 14 & 0.7428 & 26 & 313 & 1.0383 & 1029 & 0.5730 & 48 & 1105 & 1.0208 & 3641 & 0.5625 \\ 
5 & 8 & 1.2127 & 25 & 0.7001 & 27 & 338 & 1.0377 & 1113 & 0.5721 & 49 & 1152 & 1.0206 & 3798 & 0.5622 \\ 
6 & 13 & 1.1547 & 40 & 0.6667 & 28 & 365 & 1.0356 & 1200 & 0.5714 & 50 & 1201 & 1.0200 & 3958 & 0.5619 \\ 
7 & 18 & 1.1508 & 58 & 0.6472 & 29 & 392 & 1.0351 & 1291 & 0.5706 & 51 & 1250 & 1.0198 & 4121 & 0.5617 \\ 
8 & 25 & 1.1202 & 80 & 0.6305 & 30 & 421 & 1.0333 & 1385 & 0.5699 & 52 & 1301 & 1.0192 & 4287 & 0.5615 \\ 
9 & 32 & 1.1163 & 104 & 0.6225 & 31 & 450 & 1.0328 & 1482 & 0.5693 & 53 & 1352 & 1.0190 & 4457 & 0.5613 \\ 
10 & 41 & 1.0976 & 133 & 0.6120 & 32 & 481 & 1.0312 & 1583 & 0.5686 & 54 & 1405 & 1.0185 & 4630 & 0.5611 \\ 
11 & 50 & 1.0945 & 164 & 0.6064 & 33 & 512 & 1.0307 & 1687 & 0.5680 & 55 & 1458 & 1.0183 & 4807 & 0.5609 \\ 
12 & 61 & 1.0820 & 198 & 0.6023 & 34 & 545 & 1.0294 & 1794 & 0.5675 & 56 & 1513 & 1.0178 & 4987 & 0.5607 \\ 
13 & 72 & 1.0796 & 236 & 0.5977 & 35 & 578 & 1.0290 & 1904 & 0.5671 & 57 & 1568 & 1.0177 & 5170 & 0.5605 \\ 
14 & 85 & 1.0706 & 277 & 0.5943 & 36 & 613 & 1.0277 & 2018 & 0.5666 & 58 & 1625 & 1.0172 & 5356 & 0.5604 \\ 
15 & 98 & 1.0687 & 322 & 0.5906 & 37 & 648 & 1.0274 & 2135 & 0.5662 & 59 & 1682 & 1.0171 & 5546 & 0.5602 \\ 
16 & 113 & 1.0620 & 370 & 0.5878 & 38 & 685 & 1.0263 & 2256 & 0.5657 & 60 & 1741 & 1.0167 & 5738 & 0.5601 \\ 
17 & 128 & 1.0604 & 420 & 0.5862 & 39 & 722 & 1.0260 & 2379 & 0.5653 & 61 & 1800 & 1.0165 & 5935 & 0.5599 \\ 
18 & 145 & 1.0552 & 475 & 0.5837 & 40 & 761 & 1.0250 & 2506 & 0.5650 & 62 & 1861 & 1.0161 & 6134 & 0.5597 \\ 
19 & 162 & 1.0539 & 533 & 0.5817 & 41 & 800 & 1.0247 & 2637 & 0.5645 & 63 & 1922 & 1.0160 & 6337 & 0.5596 \\ 
20 & 181 & 1.0497 & 594 & 0.5800 & 42 & 841 & 1.0238 & 2770 & 0.5642 & 64 & 1985 & 1.0156 & 6543 & 0.5594 \\ 
21 & 200 & 1.0487 & 658 & 0.5787 & 43 & 882 & 1.0235 & 2907 & 0.5639 & 65 & 2048 & 1.0155 & 6753 & 0.5593 \\ 
22 & 221 & 1.0453 & 726 & 0.5772 & 44 & 925 & 1.0227 & 3047 & 0.5636 & 66 & 2113 & 1.0151 & 6965 & 0.5592 \\ 
23 & 242 & 1.0444 & 797 & 0.5759 & 45 & 968 & 1.0225 & 3191 & 0.5632 & 67 & 2178 & 1.0150 & 7182 & 0.5590 \\ 
24 & 265 & 1.0415 & 871 & 0.5749 & 46 & 1013 & 1.0217 & 3337 & 0.5630 & 68 & 2245 & 1.0147 & 7401 & 0.5589 \\
\hline
\end{tabular}
\bigskip
\caption{The table displays pairs $(k,q_-), (k,q_+)$ which are numerically $-1$- and $+1$-feasible, respectively. Recall that, according to Definition \ref{def:sfeasibility}, a pair $(k,q)$ is $s$-feasible if there exists $f\in\a^{\textup{disc}}_s(q)$, such that $k_{sf}\leq k$. We produced this table using {\it Gurobi} \cite{gurobi} and PARI/GP \cite{gp}. We have checked numerically that, for any given pair $(k,q_\pm)$ from the table, the pairs $(k',q_s), (k, q'_s)$ are always $s$-feasible, for any $k'\geq k$ and $q'_s\leq q_s$. We also verified numerically that the set of integers $q$, for which  $(k,q)$ is $s$-feasible but $(k-1,q)$ is not, coincides with the interval $[q_s(k),q_s(k+1)-1]$, where $k\mapsto q_s(k)$ is the function given by the table. Thus the table seems to indeed record the jumps of the function $q\mapsto \A^{\textup{disc}}_s(q)$.}
\label{table:1}
\end{changemargin}
\end{table}

\begin{figure}[H]
\centering
\begin{tikzpicture}[scale=1.6]
\begin{axis}[
axis lines=middle,
xtick={0,1,2,3,4,5,6,7,8,9,10},
xticklabel style={font=\tiny},
ytick={0},
ymin=-1.8,
ymax=0.8,
xmax=8,
xmin=0,
]
\addplot[
	solid,
	color=blue,
]
table{babymonster.data};
\addplot[
  solid,
]
table{babymonster2.data};
\end{axis}
\end{tikzpicture}
\caption{There are two plots. The one in blue corresponds to a plot of the sequence $\left\{\left(\tfrac{n}{\sqrt{2q+1}},f(n)\right)\right\}_{n=0}^q$, where $f$ is an optimal answer to Problem \ref{DiscFeasLP} in the case $s=-1$ with $k_{-f}=120$ and $q=\left\lceil {(k_{-f}-1)^2/2}\right\rceil=7081$. Moreover, this vector satisfies $\ft f=-f$, $f(0)=0$, and has minimal energy $\sum_{n=120}^{7081} f(n)^2$. This plot almost matches the plot of the function $f_\star(x)=\sin(2\pi |x|) {\bf 1}_{[-1,1]}(x)-\frac{2\sin^2(\pi x)}{\pi(1-x^2)}$ (in black) which was included for comparison.}
\label{figure:2}
\end{figure}

 \subsection{Discrete Hankel Transform}\label{numerics:DHT}
 Tables \ref{table:2} and \ref{table:3} display numerical data\footnote{The main reason to display Tables \ref{table:2}, \ref{table:3} in full is that it might be possible to spot certain numerical patterns and thus produce conjectures towards the continuous sign uncertainty constants $\A_s(d)$ for dimensions other than $d\in\{1,2,8,12,24\}$.} relative to the sign uncertainty principles for the discrete Hankel transform. 
For each sign $s\in\{+,-\}$, dimension $d$, and parameter $k$, the pair $(k,q_s)$ is numerically $(s,\tfrac{d}2-1)$-feasible, in the sense of Definition \ref{def:snufeasibility}. We used floating-point arithmetic, and therefore we cannot claim these numbers to be correct in the theoretical sense, but we believe they are. We have checked numerically that, for any given pair $(k,q_s)$ in these tables, the pairs $(k',q_s), (k, q'_s)$ are always $s$-feasible, for any $k'\geq k$ and $q'_s\leq q_s$. We have also numerically verified that the set of integers $q$, for which $(k,q)$ is $(s,\tfrac{d}2-1)$-feasible but $(k-1,q)$ is not, coincides with the interval $[q_s(k;d),q_s(k+1;d)-1]$, where $k\mapsto q_s(k;d)$ denotes the function given by Tables  \ref{table:2} and \ref{table:3}. Hence these tables seem to record the jumps of the function $q\mapsto \A^{\textup{disc}}_s(q,\tfrac{d}2-1)$. 

It does not seem easy to detect any distinguishable patterns in the entries of Tables \ref{table:2} and \ref{table:3}, except for the special cases $d\in\{2,8,24\}$ when $s=-1$, and $d=12$ when $s=+1$. In these cases, one can indeed spot a pattern in the first few entries of the corresponding columns, which in turn motivated Conjecture \ref{conj:discHTfeaspairs}. If $(s,d)=(-,2)$, then the sequence
\begin{equation}\label{eq:seqdef2}
\left\lfloor\frac{\sqrt{3}(k^2-2k+2)}{4}\right\rfloor_{k\geq 4} = 4, 7, 11, 16, 21, 28, 35, 43, 52, 62,\ldots
\end{equation}
matches the data from Table \ref{table:2} for $k\in\{4,5,6,7,8\}$, and seems to be slightly below the values from that table if $k>8$. In particular, this means that $\left(k,\lfloor\frac{\sqrt{3}(k^2-2k+2)}{4}\rfloor\right)$ should be $(s,2/1-1)$-feasible, for all $k\geq 4$. Similarly, if $(s,d)=(-,8),(-,24),(+,12)$ respectively, then the data match the sequences\footnote{From the available data, one could try to look for a best-fitting quadratic polynomial whose floor function agrees with the data for many more values of $k$. Our choice was the simplest one among those with rational coefficients and small denominators. }
\begin{align}\label{eq:seqdef82412}
\begin{split}
\left\lfloor\frac{k^2}{4}\right\rfloor_{k\geq 4}&=4, 6, 9, 12, 16, 20, 25, 30, 36, 42,\ldots,\\
\left\lfloor\frac{k^2+6k-8}{8}\right\rfloor_{k\geq 4}&=4, 5, 8, 10, 13, 15, 19, 22, 26, 29,\ldots,\\
\left\lfloor\frac{k^2+2k-1}{4}\right\rfloor_{k\geq 3}&=3, 5, 8, 11, 15, 19, 24, 29, 35, 41,\ldots,
\end{split}
\end{align}
 for $k\in\{4,5,6,7,8,9,10,11,12\}$, $k\in\{4,5,6,7,8\}$, and $k\in\{3,4,5,6,7,8,9,10,11\}$. 

Similarly to what was already observed in \S\ref{numerics:DFT}, the $+1$ problem seems to be computationally harder than the $-1$ problem. Nevertheless, one can check that the sequences in \eqref{eq:seqdef2} and \eqref{eq:seqdef82412}  always belong to the interval $(q_s(k-1;d),q_s(k;d)]$ for $k\leq 30$ and $(s,d)\in\{(-,2),(-,8),(-,24),(+,12)\}$, respectively. This means that $k-1$ coincides with the quantities
\begin{align}
\begin{split}
& \A^{\textup{disc}}_-\left(\left\lfloor\frac{\sqrt{3}(k^2-2k+2)}{4}\right\rfloor,\frac{2}{2}-1\right) ,\ \
\A^{\textup{disc}}_-\left(\left\lfloor\frac{k^2}{4}\right\rfloor ,\frac{8}{2}-1\right) , \\
& \A^{\textup{disc}}_-\left(\left\lfloor\frac{k^2+6k-8}{8}\right\rfloor,\frac{24}{2}-1\right), \ \
\A^{\textup{disc}}_+\left(\left\lfloor\frac{k^2+2k-1}{4}\right\rfloor,\frac{12}{2}-1\right),
\end{split}
\end{align}
and provides further evidence towards Conjecture \ref{conj:discHTfeaspairs}.

\newpage

\begin{landscape}
\begin{table}[!ht]
\centering
\begin{threeparttable}
\begin{tabular}{|c|ccccccccccccccccccccccc|}
\hline
\backslashbox{$k$}{$d$}&2& 3& 4& 5& 6& 7& 8& 9& 10& 11& 12& 13& 14& 15& 16& 17& 18& 19& 20& 21& 22& 23& 24\\
\hline
4 & 4& 4& 4& 4& 4& 4& 4& 4& 4& 4& 4& 4& 4& 4& 4& 4& 4& 4& 4& 4& 4& 4& 4\\
5 & 7& 7& 7& 7& 6& 6& 6& 6& 6& 6& 6& 6& 6& 6& 5& 5& 5& 5& 5& 5& 5& 5& 5\\
6 & 11& 9& 9& 9& 9& 9& 9& 8& 8& 8& 8& 8& 8& 8& 8& 8& 8& 8& 8& 8& 8& 8& 8\\
7 & 16& 15& 15& 14& 13& 13& 12& 12& 11& 11& 11& 11& 11& 11& 11& 11& 11& 11& 11& 11& 11& 11& 10\\
8 & 21& 19& 18& 17& 17& 16& 16& 15& 15& 15& 15& 14& 14& 14& 14& 14& 14& 14& 14& 14& 14& 13& 13\\
9 & 29& 27& 25& 23& 22& 21& 20& 20& 19& 19& 19& 18& 18& 18& 18& 17& 17& 17& 17& 17& 17& 17& 16\\
10 & 35& 31& 29& 28& 27& 26& 25& 24& 24& 23& 23& 23& 22& 22& 22& 21& 21& 21& 21& 20& 20& 20& 20\\
11 & 45& 41& 38& 35& 33& 31& 30& 30& 29& 28& 27& 27& 26& 26& 26& 25& 25& 25& 25& 24& 24& 24& 24\\
12 & 53& 47& 43& 41& 39& 38& 36& 35& 34& 33& 33& 32& 31& 31& 30& 30& 29& 29& 29& 28& 28& 28& 28\\
13 & 64& 58& 53& 49& 46& 44& 43& 41& 40& 39& 38& 37& 36& 36& 35& 35& 34& 33& 33& 33& 32& 32& 32\\
14 & 74& 65& 60& 57& 54& 52& 49& 48& 46& 45& 44& 43& 42& 41& 40& 40& 39& 38& 38& 37& 37& 36& 36\\
15 & 87& 79& 72& 66& 62& 59& 57& 55& 53& 52& 50& 49& 48& 47& 46& 45& 44& 44& 43& 42& 42& 41& 41\\
16 & 98& 87& 80& 75& 71& 68& 65& 62& 60& 58& 57& 55& 54& 53& 52& 51& 50& 49& 48& 48& 47& 46& 46\\
17 & 114& 102& 93& 85& 80& 76& 73& 70& 68& 66& 64& 62& 61& 59& 58& 57& 56& 55& 54& 53& 52& 52& 51\\
18 & 126& 111& 102& 96& 90& 86& 82& 79& 76& 73& 71& 69& 67& 66& 64& 63& 62& 61& 60& 59& 58& 57& 56\\
19 & 143& 129& 117& 107& 101& 96& 91& 88& 84& 81& 79& 77& 75& 73& 71& 70& 68& 67& 66& 65& 64& 63& 62\\
20 & 157& 139& 128& 119& 112& 106& 101& 97& 93& 90& 87& 85& 82& 80& 78& 77& 75& 74& 72& 71& 70& 69& 68\\
21 & 177& 158& 143& 132& 124& 117& 112& 107& 103& 99& 96& 93& 90& 88& 86& 84& 82& 81& 79& 78& 76& 75& 74\\
22 & 192& 169& 155& 145& 136& 129& 123& 117& 113& 108& 105& 102& 99& 96& 94& 92& 90& 88& 86& 85& 83& 82& 80\\
23 & 213& 191& 173& 159& 149& 141& 134& 128& 123& 118& 114& 111& 108& 105& 102& 99& 97& 95& 94& 92& 90& 89& 87\\
24 & 231& 203& 186& 173& 162& 153& 146& 139& 134& 129& 124& 120& 117& 113& 111& 108& 105& 103& 101& 99& 97& 96& 94\\
25 & 254& 227& 205& 188& 176& 166& 158& 151& 145& 139& 134& 130& 126& 122& 119& 116& 114& 111& 109& 107& 105& 103& 101\\
26 & 272& 240& 220& 204& 191& 180& 171& 163& 156& 150& 145& 140& 136& 132& 128& 125& 122& 120& 117& 115& 112& 111& 109\\
27 & 297& 266& 239& 220& 206& 194& 184& 176& 168& 162& 156& 151& 146& 142& 138& 134& 131& 128& 126& 123& 121& 118& 116\\
28 & 318& 280& 256& 237& 222& 209& 198& 189& 181& 174& 167& 162& 157& 152& 148& 144& 140& 137& 134& 132& 129& 127& 124\\
29 & 344& 308& 277& 255& 238& 224& 213& 203& 194& 186& 179& 173& 168& 163& 158& 154& 150& 147& 143& 140& 138& 135& 133\\
30 & 367& 323& 295& 273& 255& 240& 228& 217& 207& 199& 191& 185& 179& 173& 168& 164& 160& 156& 153& 149& 146& 144& 141\\
\hline
\end{tabular}
\caption{Numerical data for the discrete Hankel transform $-1$ uncertainty principle. If $q_-$ is an entry in the table, then $(k,q_-)$ is numerically $(-1,\tfrac{d}2-1)$-feasible. The {\it Gurobi} solver \cite{gurobi} was used with PARI/GP \cite{gp} as interface.}\label{table:2}
\end{threeparttable}
\end{table}
\end{landscape}

\begin{landscape}
\begin{table}[!h]
\begin{tabular}{|c|ccccccccccccccccccccccc|}
\hline
\backslashbox{$k$}{$d$}&
2& 3& 4& 5& 6& 7& 8& 9& 10& 11& 12& 13& 14& 15& 16& 17& 18& 19& 20& 21& 22& 23& 24\\
\hline
3 & 3& 3& 3& 3& 3& 3& 3& 3& 3& 3& 3& 3& 3& 3& 3& 3& 3& 3& 3& 3& 3& 3& 3\\
4 & 9& 7& 7& 6& 6& 6& 6& 5& 5& 5& 5& 5& 5& 5& 5& 5& 5& 5& 5& 5& 5& 5& 4\\
5 & 14& 12& 10& 9& 9& 8& 8& 8& 8& 8& 8& 8& 7& 7& 7& 7& 7& 7& 7& 7& 7& 7& 7\\
6 & 24& 18& 16& 15& 14& 13& 12& 12& 12& 11& 11& 11& 11& 10& 10& 10& 10& 10& 10& 10& 10& 10& 10\\
7 & 33& 24& 21& 19& 18& 17& 17& 16& 16& 15& 15& 15& 14& 14& 14& 14& 13& 13& 13& 13& 13& 13& 13\\
8 & 46& 35& 30& 27& 24& 23& 22& 21& 20& 20& 19& 19& 18& 18& 18& 18& 17& 17& 17& 17& 16& 16& 16\\
9 & 58& 42& 36& 33& 31& 29& 28& 27& 26& 25& 24& 23& 23& 22& 22& 22& 21& 21& 21& 21& 20& 20& 20\\
10 & 75& 56& 48& 42& 39& 36& 34& 33& 31& 30& 29& 29& 28& 27& 27& 26& 26& 25& 25& 25& 24& 24& 24\\
11 & 90& 66& 56& 51& 47& 44& 42& 40& 38& 37& 35& 34& 33& 33& 32& 31& 31& 30& 30& 29& 29& 28& 28\\
12 & 111& 82& 70& 61& 56& 52& 50& 47& 45& 43& 42& 40& 39& 38& 37& 37& 36& 35& 35& 34& 33& 33& 32\\
13 & 129& 94& 80& 72& 66& 62& 58& 55& 53& 50& 49& 47& 46& 44& 43& 42& 41& 40& 40& 39& 39& 38& 37\\
14 & 153& 114& 97& 85& 77& 72& 67& 64& 61& 58& 56& 54& 52& 51& 50& 48& 47& 46& 45& 44& 44& 43& 43\\
15 & 175& 127& 109& 97& 89& 83& 77& 73& 70& 66& 64& 62& 60& 58& 56& 55& 54& 52& 51& 50& 49& 48& 48\\
16 & 203& 150& 128& 111& 102& 94& 88& 83& 79& 75& 72& 69& 67& 65& 63& 62& 60& 59& 58& 56& 55& 54& 54\\
17 & 229& 166& 141& 126& 115& 106& 99& 93& 89& 85& 81& 78& 75& 73& 71& 69& 67& 65& 64& 63& 62& 61& 60\\
18 & 260& 192& 163& 142& 129& 119& 111& 105& 99& 95& 90& 87& 84& 81& 79& 77& 75& 73& 71& 69& 68& 67& 66\\
19 & 289& 209& 178& 159& 144& 133& 124& 116& 110& 105& 100& 96& 93& 90& 87& 84& 82& 80& 78& 77& 75& 74& 72\\
20 & 324& 238& 202& 177& 160& 147& 137& 129& 122& 116& 111& 106& 102& 99& 96& 93& 90& 88& 86& 84& 82& 81& 79\\
21 & 355& 258& 220& 195& 177& 163& 151& 142& 134& 127& 122& 117& 112& 108& 105& 102& 99& 96& 94& 92& 90& 88& 86\\
22 & 395& 290& 245& 215& 194& 179& 166& 155& 147& 139& 133& 127& 122& 118& 114& 111& 108& 105& 102& 100& 97& 96& 94\\
23 & 430& 311& 265& 235& 213& 195& 181& 170& 160& 152& 145& 139& 133& 128& 124& 120& 117& 114& 111& 108& 106& 103& 101\\
24 & 472& 347& 293& 257& 232& 213& 197& 185& 174& 165& 157& 150& 145& 139& 134& 130& 126& 123& 120& 117& 114& 112& 109\\
25 & 511& 370& 315& 279& 252& 231& 214& 200& 188& 179& 170& 163& 156& 150& 145& 140& 136& 133& 129& 126& 123& 120& 118\\
26 & 558& 409& 345& 302& 273& 250& 231& 216& 204& 193& 183& 175& 168& 162& 156& 151& 147& 142& 139& 135& 132& 129& 126\\
27 & 600& 434& 369& 326& 294& 269& 249& 233& 219& 207& 197& 188& 181& 174& 168& 162& 157& 153& 149& 145& 141& 138& 135\\
28 & 649& 476& 401& 352& 317& 290& 268& 250& 235& 223& 212& 202& 194& 186& 180& 174& 168& 163& 159& 155& 151& 148& 144\\
29 & 695& 503& 427& 378& 340& 311& 288& 268& 252& 238& 227& 216& 207& 199& 192& 186& 179& 174& 170& 165& 161& 157& 154\\
30 & 748& 548& 462& 405& 364& 333& 308& 287& 270& 255& 242& 231& 221& 212& 205& 198& 191& 186& 181& 176& 171& 167& 163\\
\hline
\end{tabular}
\caption{If $q_+$ is an entry in the table, then $(k,q_+)$ is numerically $(+1,\tfrac{d}2-1)$-feasible. }\label{table:3}
\end{table}
\end{landscape}


\section*{Acknowledgments} 
The authors are grateful to Henry Cohn, David de Laat, and Danylo Radchenko for helpful discussions, and to 
the anonymous referee for a careful reading and valuable suggestions.

\bibliographystyle{amsplain}

\begin{thebibliography}{99}


\bibitem{BB09}
\textsc{E. Bannai and E. Bannai},
\newblock{\it A survey on spherical designs and algebraic combinatorics on spheres.}  
\newblock European J. Combin. {\bf 30} (2009), no.~6, 1392--1425.

\bibitem{BD79}
\textsc{E. Bannai and R. M. Damerell}, 
{\it Tight spherical designs. I.} 
J. Math. Soc. Japan {\bf 31} (1979), 199--207.

\bibitem{BD80} 
\textsc{E. Bannai and R. M. Damerell}, 
{\it Tight spherical designs. II.}
J. London Math. Soc. {\bf 21} (1980), 13--30.

\bibitem{gp}
\textsc{C. Batut, K. Belabas, D. Benardi, H. Cohen, and M. Olivier},
 \emph{User's Guide to {PARI-GP}}, version 2.11.1 (2018).

\bibitem{BCK10}
\textsc{J. Bourgain, L. Clozel, and J.-P. Kahane}, 
\newblock {\it Principe d'Heisenberg et fonctions positives.}
\newblock Ann. Inst. Fourier (Grenoble) {\bf 60} (2010), no.~4, 1215--1232. 

\bibitem{CMS19}
\newblock \textsc{E. Carneiro, M. B. Milinovich, and K. Soundararajan}, 
\newblock \emph{Fourier optimization and prime gaps.}
\newblock Comment. Math. Helv. {\bf 94} (2019), no.~3, 533--568. 

\bibitem{ACHLT20}
\newblock\textsc{N. Afkhami-Jeddi, H. Cohn, T. Hartman, D. de Laat, and A. Tajdini},
\newblock\emph{High-dimensional sphere packing and the modular bootstrap}, 
\newblock J. High Energ. Phys. 2020, {\bf 66} (2020). 

\bibitem{CE03}
	\textsc{H. Cohn and N. Elkies}, 
	\newblock{\it New upper bounds on sphere packings I.} 
	\newblock Ann. of Math. (2) {\bf 157} (2003), no.~2, 689--714. 
	
\bibitem{CG19}
\textsc{H. Cohn and F. Gon\c{c}alves}, 
\newblock {\it An optimal uncertainty principle in twelve dimensions via modular forms.}
\newblock Invent. Math. {\bf 217} (2019), no.~3, 799--831. 

\bibitem{CK07}
	\textsc{H. Cohn and A. Kumar},
\newblock	{\it Universally optimal distribution of points on spheres.}
\newblock J. Amer. Math. Soc. {\bf 20} (2007), no.~1, 99--148.

\bibitem{CKMRV17}
\textsc{H. Cohn, A. Kumar, S. Miller, D. Radchenko, and M. Viazovska},
\newblock {\it The sphere packing problem in dimension 24.}
\newblock Ann. of Math. (2) {\bf 185} (2017), no.~3, 1017--1033. 

\bibitem{CZ14}
\textsc{H. Cohn and Y. Zhao, }
\newblock {\it Sphere packing bounds via spherical codes.}
\newblock Duke Math. J. {\bf 163} (2014), no.~10, 1965--2002.

\bibitem{DX13}
\textsc{F. Dai and Y. Xu},
\newblock{Approximation theory and harmonic analysis on spheres and balls}.
\newblock Springer Monographs in Mathematics. Springer, New York, 2013. 

\bibitem{DGS77}
\textsc{P. Delsarte, J. M. Goethals, and J. J. Seidel}, 
{\it Spherical codes and designs.}
Geom. Dedicata {\bf 6} (1977), 363--388.

\bibitem{F87}
\textsc{H. Fisk Johnson},
\newblock{\it An improved method for computing a discrete Hankel transform.} 
\newblock Comput. Phys. Comm. {\bf 43} (1987), no.~2, 181--202.

\bibitem{FS97}
\textsc{G. B. Folland and A. Sitaram},
\newblock {\it The uncertainty principle: a mathematical survey.}
\newblock J. Fourier Anal. Appl. {\bf 3} (1997), no. 3, 207--238.

\bibitem{G70}
\textsc{G. Gasper},
{\it Linearization of the product of Jacobi polynomials. I.}
Canadian J. Math. {\bf 22} (1970), no. ~1, 171--175.

\bibitem{GOSR20}
\textsc{F. Gon\c{c}alves, D. Oliveira e Silva, and J. P. G. Ramos}, 
\newblock {\it On regularity and mass concentration phenomena for the sign uncertainty principle.}
\newblock  J.\@ Geom.\@ Anal.\@ {\bf 31} (2021), no.~6, 6080--6101.  


\bibitem{GOSS17}
\textsc{F. Gon\c{c}alves, D. Oliveira e Silva, and S. Steinerberger}, 
\newblock {\it Hermite polynomials, linear flows on the torus, and an uncertainty principle for roots.} 
\newblock J. Math. Anal. Appl. {\bf 451} (2017), no.~2, 678--711. 

\bibitem{GOSS19}
\textsc{F. Gon\c{c}alves, D. Oliveira e Silva, and S. Steinerberger}, 
\newblock {\it A universality law for sign correlations of eigenfunctions of differential operators.}
\newblock  J. Spectr. Theory (2021), 1--16, DOI 10.4171/JST/351.

\bibitem{GVT19}
\textsc{D. V. Gorbachev, V. I. Ivanov, and S. Yu. Tikhonov},
\newblock {\it Uncertainty principles for eventually constant sign bandlimited functions.}
\newblock SIAM J. Math. Anal. 52 (2020), no.~5, 4751--4782. 

\bibitem{Gr08}
\textsc{L. Grafakos}, 
\newblock{ Classical Fourier analysis.} 
\newblock Second edition. Graduate Texts in Mathematics, 249. Springer, New York, 2008.

\bibitem{H00}
\textsc{T. C. Hales}, 
\newblock{\it Cannonballs and honeycombs.}
\newblock Notices Amer. Math. Soc. {\bf 47} (2000), no.~4, 440--449.

\bibitem{gurobi}
\textsc{Gurobi Optimization, LLC},
\newblock {\it Gurobi Optimizer Reference Manual} (2020).


\bibitem{KL78}
\textsc{G. A. Kabatiansky and V. I. Levenshtein}, 
\newblock{\it Bounds for packings on a sphere and in space (in Russian).}
\newblock Problemy Peredachi Informacii {\bf 14} (1978), 3--25; 
\newblock English translation in Probl. Inf. Transm. {\bf 14} (1978), 1--17.



\bibitem{OST17}
\textsc{D. Oliveira e Silva and C. Thiele},
\newblock{\it Estimates for certain integrals of products of six Bessel functions.} 
\newblock Rev. Mat. Iberoam. {\bf 33} (2017), no.~4, 1423--1462.

\bibitem{oeis}
\newblock{\it  On-Line Encyclopedia of Integer Sequences}.
\newblock \url{https://oeis.org}.

\bibitem{Pi72}
\textsc{S. K. Pichorides}, 
\newblock{\it On the best values of the constants in the theorems of M. Riesz, Zygmund and Kolmogorov.}
\newblock Studia Math. {\bf 44} (1972), 165--179.

\bibitem{Ro55}
\textsc{H. Robbins}, 
\newblock {\it A remark on Stirling's formula.} 
\newblock Amer. Math. Monthly {\bf 62} (1955), 26--29. 

\bibitem{R62}
\textsc{W. Rudin},
\newblock Fourier analysis on groups. 
\newblock Interscience Tracts in Pure and Applied Math., no.~12. Wiley, New York, 1962.

\bibitem{Sz75}
\textsc{G. Szeg\"o},
\newblock{Orthogonal polynomials.} 
\newblock Fourth edition. American Mathematical Society, Colloquium Publications, Vol. XXIII. American Mathematical Society, Providence, R.I., 1975.

\bibitem{T37}
\textsc{E. C. Titchmarsh},
\newblock Introduction to the Theory of Fourier Integrals. 
\newblock Chelsea Publishing Co., New York, 1986. 

\bibitem{Vi17}
\textsc{M. Viazovska},
\newblock {\it The sphere packing problem in dimension 8.}
\newblock Ann. of Math. (2) {\bf 185} (2017), no.~3, 991--1015. 

\bibitem{W66}
\textsc{G. N. Watson},
\newblock A Treatise on the Theory of Bessel Functions. 
\newblock Cambridge University Press, Cambridge, 1966.

\end{thebibliography}


\begin{dajauthors}
\begin{authorinfo}[fg]
  Felipe Gon\c{c}alves\\
  Instituto Nacional de Matem\'atica Pura e Aplicada\\
  Rio de Janeiro, Brazil\\
  goncalves\imageat{}impa\imagedot{}br \\
  \url{https://w3.impa.br/%7Egoncalves/index.html}
\end{authorinfo}
\begin{authorinfo}[dos]
  Diogo Oliveira e Silva\\
  Instituto Superior Técnico\\
  Lisboa, Portugal\\
  diogo.oliveira.e.silva\imageat{}tecnico\imagedot{}ulisboa\imagedot{}pt \\
  \url{https://www.math.tecnico.ulisboa.pt/~oliveiraesilva/}
\end{authorinfo}
\begin{authorinfo}[jpgr]
  Jo\~ao P. G. Ramos\\
    Eidgenössische Technische Hochschule\\
  Zürich, Switzerland\\
  joao.ramos\imageat{}math\imagedot{}ethz\imagedot{}ch\\
  \url{https://sites.google.com/view/gionnoramos/}
\end{authorinfo}
\end{dajauthors}

\end{document}